\documentclass[10pt,a4paper]{article}
\usepackage{amsmath}
\usepackage{amsfonts}
\usepackage{amssymb}
\usepackage{amsthm}
\usepackage{graphicx}
\usepackage[left=2.8cm,right=2.8cm,top=3cm,bottom=3cm]{geometry}
\usepackage{verbatim}
\usepackage{tikz}
\usetikzlibrary{calc,angles,quotes,hobby,decorations.markings}
\usepackage{epsfig}
\usepackage{xstring}
\usepackage{intcalc}
\usepackage{comment}
\usepackage[titletoc]{appendix}
\usepackage{url, hyperref}
\hypersetup{
    colorlinks=true,
    linkcolor=blue,
    filecolor=magenta,      
    urlcolor=red,
    citecolor=red,
}
\usepackage{cite}
\usepackage{enumitem}
\usepackage{color}

\newcommand{\Addresses}{{
  \bigskip
  
  \textsc{Laboratoire Alexander Grothendieck, Institut des Hautes \'Etudes Scientifiques, Universit\'e Paris-Saclay, 35 route de Chartres, 91440 Bures-sur-Yvette, France}\par
  \textit{E-mail address}: \texttt{pierre-louis.blayac@u-psud.fr}
}}

\newcommand{\R}{\mathbb{R}}
\newcommand{\Q}{\mathbb{Q}}
\newcommand{\C}{\mathbb{C}}
\newcommand{\Z}{\mathbb{Z}}
\newcommand{\N}{\mathbb{N}}
\newcommand{\K}{\mathbb{K}}
\newcommand{\GL}{\mathrm{GL}}
\newcommand{\SL}{\mathrm{SL}}

\newcommand{\PGL}{\mathrm{PGL}}
\newcommand{\PSL}{\mathrm{PSL}}
\newcommand{\PR}{\mathrm{P}}

\newcommand{\Ccal}{\mathcal{C}}

\newtheorem{prop}{Proposition}[section]
\newtheorem{thm}[prop]{Theorem}

\newtheorem{lemma}[prop]{Lemma}
\newtheorem{fait}[prop]{Fact}
\newtheorem{cor}[prop]{Corollary}
\newtheorem{obs}[prop]{Observation}

\theoremstyle{definition}
\newtheorem{defi}[prop]{Definition}

\theoremstyle{remark}
\newtheorem{rqq}[prop]{Remark}
\newtheorem{nota}[prop]{Notation}

\numberwithin{equation}{section}

\DeclareMathOperator{\Span}{Span}
\DeclareMathOperator{\Ker}{Ker}
\DeclareMathOperator{\Aut}{Aut}
\DeclareMathOperator{\End}{End}
\DeclareMathOperator{\axis}{Axis}

\DeclareMathOperator{\Geod}{Geod}
\DeclareMathOperator{\NW}{NW}

\newcommand{\ie}{i.e.\ }
\newcommand{\eg}{e.g.\ }
\newcommand{\resp}{resp.\ }

\newcounter{mycount}

\newcommand{\length}[1]{
    \setcounter{mycount}{0}
    \foreach \x in {#1}{
        \stepcounter{mycount}
        }
}

\makeatletter
\newcommand{\listnb}[3]{
    \foreach \temp@a [count=\temp@i] in {#1} {
        \IfEq{\temp@i}{#2}{\global\let#3\temp@a\breakforeach}{}
    }
    \par
}
\makeatother

\newcommand{\listnbmod}[2]{
\listnb{#1}{\intcalcInc{\intcalcMod{\intcalcDec{#2}}{\themycount}}}
}

\newcommand{\getanglepoints}[3]{
    \pgfmathanglebetweenpoints{\pgfpointanchor{#2}{center}}
                              {\pgfpointanchor{#3}{center}}
    \global\let#1\pgfmathresult  
}

\newcommand{\getanglelines}[5]{
    \pgfmathanglebetweenlines{\pgfpointanchor{#2}{center}}
                             {\pgfpointanchor{#3}{center}}
                             {\pgfpointanchor{#4}{center}}
                             {\pgfpointanchor{#5}{center}}
    \global\let#1\pgfmathresult  
}

\makeatletter
\newcommand{\getdistance}[3]{
  \pgfpointdiff{\pgfpointanchor{#2}{center}} 
               {\pgfpointanchor{#3}{center}} 
  \pgf@xa=\pgf@x
  \pgf@ya=\pgf@y
  \pgfmathparse{veclen(\pgf@xa,\pgf@ya)/28.45274/2} 
  \global\let#1\pgfmathresult 
}
\makeatother

\makeatletter
\newcommand{\chooseangle}[6]{
\getanglelines{\temp@a}{#1}{#2}{#2}{#3}
\getanglelines{\temp@b}{#2}{#3}{#3}{#4}
\getanglelines{\temp@c}{#3}{#4}{#4}{#5}
\getanglepoints{\temp@d}{#3}{#4}
\pgfmathsetmacro\temp@e{ifthenelse(greater(\temp@a +\temp@c ,0.001),-\temp@c /(max(\temp@a +\temp@c ,0.001))*\temp@b+\temp@d,-\temp@b /2 +\temp@d)}
\global\let#6=\temp@e
}
\makeatother

\makeatletter
\newcommand{\choosecontrolpoints}[7]{
\getanglepoints{\temp@c}{#1}{#2}
\getdistance{\temp@l}{#1}{#2}
\pgfmathsetmacro\temp@a{Mod(\temp@c -#3 ,360)}
\pgfmathsetmacro\temp@b{Mod(#4-\temp@c ,360)}
\pgfmathsetmacro\temp@la{
    ifthenelse(less(\temp@a,90),
        ifthenelse(less(\temp@b,90),abs(sin(\temp@b))*\temp@l,\temp@l),
        ifthenelse(less(\temp@b,90),abs(sin(\temp@b))*\temp@l,\temp@l))}
\pgfmathsetmacro\temp@lb{
    ifthenelse(less(\temp@b,90),
        ifthenelse(less(\temp@a,90),abs(sin(\temp@a))*\temp@l,\temp@l),
        ifthenelse(less(\temp@a,90),abs(sin(\temp@a))*\temp@l,\temp@l))}
        \pgfmathparse{#7*\temp@la}
\global\let#5=\pgfmathresult
        \pgfmathparse{#7*\temp@lb}
\global\let#6=\pgfmathresult
}
\makeatother

\makeatletter
\newcommand{\cvx}[2]{
\foreach \@i in {1,...,\themycount} {

\listnbmod{#1}{\intcalcSub{\@i}{2}}{\@A}
\listnbmod{#1}{\intcalcSub{\@i}{1}}{\@B}
\listnbmod{#1}{\@i}{\@C}
\listnbmod{#1}{\intcalcAdd{\@i}{1}}{\@D}
\listnbmod{#1}{\intcalcAdd{\@i}{2}}{\@E}
\listnbmod{#1}{\intcalcAdd{\@i}{3}}{\@F}

\chooseangle{\@A}
            {\@B}
            {\@C}
            {\@D}
            {\@E}
            {\@a}
\chooseangle{\@B}
            {\@C}
            {\@D}
            {\@E}
            {\@F}
            {\@b}
\choosecontrolpoints{\@C}
                    {\@D}
                    {\@a}
                    {\@b}
                    {\@la}
                    {\@lb}
                    {#2}
                    
\coordinate (G) at ($(\@C)+(\@a:\@la)$);
\coordinate (H) at ($(\@D)+(\@b+180:\@lb)$);

\draw (\@C) .. controls (G) and (H) .. (\@D);
}
}
\makeatother

\makeatletter 
\newcommand{\cvxx}[2]{
\foreach \@i in {1,...,\themycount} {

\listnbmod{#1}{\intcalcSub{\@i}{2}}{\@A}
\listnbmod{#1}{\intcalcSub{\@i}{1}}{\@B}
\listnbmod{#1}{\@i}{\@C}
\listnbmod{#1}{\intcalcAdd{\@i}{1}}{\@D}
\listnbmod{#1}{\intcalcAdd{\@i}{2}}{\@E}
\listnbmod{#1}{\intcalcAdd{\@i}{3}}{\@F}

\chooseangle{\@A}
            {\@B}
            {\@C}
            {\@D}
            {\@E}
            {\@a}
\chooseangle{\@B}
            {\@C}
            {\@D}
            {\@E}
            {\@F}
            {\@b}
\choosecontrolpoints{\@C}
                    {\@D}
                    {\@a}
                    {\@b}
                    {\@la}
                    {\@lb}
                    {#2}
                    
\coordinate (G) at ($(\@C)+(\@a:\@la)$);
\coordinate (H) at ($(\@D)+(\@b+180:\@lb)$);
\coordinate (A\@i) at (intersection of \@C--G and \@D--H);

\draw (\@C) .. controls (G) and (H) .. (\@D);
}
}
\makeatother

\makeatletter 
\newcommand{\draww}[7]{
\coordinate (@A) at ($(#1)!-#3!(#2)$);
\coordinate (@B) at ($(#2)!-#4!(#1)$);
\draw[#5] (@A)--(@B) node[#6]{#7};
}
\makeatother

\title{Topological mixing of the geodesic flow on convex projective manifolds}
\author{Pierre-Louis Blayac}
\date{}

\begin{document}

\maketitle

\begin{abstract}
 We introduce a natural subset of the unit tangent bundle of a convex projective manifold, the biproximal unit tangent bundle; it is closed and invariant under the geodesic flow, and we prove that the geodesic flow is topologically mixing on it whenever the manifold is irreducible. We also show that, for higher-rank, irreducible, compact convex projective manifolds, the geodesic flow is topologically mixing on each connected component of the non-wandering set.
\end{abstract}

\section{Introduction}

This article is concerned with \emph{convex projective manifolds}, namely quotients $M=\Omega/\Gamma$ of a properly convex open subset $\Omega$ of a finite-dimensional real projective space $\PR(V)$ by a torsion-free discrete subgroup $\Gamma$ of $\PGL(V)$ preserving $\Omega$. Recall that \emph{properly convex} means that $\Omega$ is convex and bounded in some affine chart of $\PR(V)$. These manifolds are generalisations of real hyperbolic manifolds, to whom they bring a new diversity of geometric features. When $M$ is compact, we say that $\Gamma$ \emph{divides} $\Omega$ and that $\Omega$ is a \emph{divisible convex set} (see \cite{benoist_survey}).

Convex projective manifolds are endowed with a natural Finsler metric, which is not necessarily Riemannian. This metric defines a \emph{geodesic flow} $(\phi_t)_{t\in\R}$ on the unit tangent bundle $T^1M=T^1\Omega/\Gamma$, obtained by following geodesics contained in projective lines (see Section~\ref{distance}). 

There is a dichotomy depending on whether $\Omega$ is \emph{strictly convex} (meaning there is no non-trivial segment in the boundary $\partial\Omega$ of $\Omega$ in $\PR(V)$, see Section~\ref{extreme_yet_smooth}), or not.

On the one hand, when $\Omega$ is strictly convex and $\Omega/\Gamma$ is compact, Benoist \cite[Th.\,1.1]{CD1} proved that many dynamical properties of the classical geodesic flow on compact hyperbolic manifolds still hold. Then Crampon and Marquis \cite{CM2014flot} generalised this to the case when $\Omega$ is strictly convex but $\Omega/\Gamma$ is not necessarily compact.

On the other hand, Benoist proved that when $\Omega$ is \emph{not} strictly convex, one of the key properties of the classical geodesic flow, namely uniform hyperbolicity, is never satisfied. Still, Bray \cite[Th.\,5.7]{bray_top_mixing} managed to recover some classical dynamical properties of the geodesic flow in this context: namely, he established that the geodesic flow is \emph{topologically transitive}, and even \emph{topologically mixing} when $\Omega/\Gamma$ is compact and 3-dimensional, not necessarily strictly convex, and $\Gamma$ is strongly irreducible (\ie $\Gamma$ does not preserve any finite union of proper projective subspaces). Recall that a continuous flow $(f_t)_{t\in\R}$ on a topological space $X$ is called topologically transitive if for any non-empty open subsets $U,V\subset X$, there exists $t>0$ such that $f_t(U)$ meets $V$, and $(f_t)_t$ is called topologically mixing if for any non-empty open $U,V\subset X$ and any large enough $t>0$, the set $f_t(U)$ meets $V$; note that the latter property implies the former. In order to prove his theorem, Bray used --- and this is where the assumption that $\Omega/\Gamma$ is compact and 3-dimensional is crucial --- another paper of Benoist \cite[Th.\,1.1]{CD4}, which gives a precise and beautiful description of these compact $3$-manifolds.

In this paper, we generalise Bray's result on topological mixing to the setting where $\Omega/\Gamma$ is not necessarily compact, and to arbitrary dimension, as we explain below.

More refined dynamical properties of the geodesic flow on convex projective manifolds will be established in the forthcoming paper \cite{mesureBM}. In particular, we shall generalise \cite{bray_ergodicity} and prove the existence of a unique flow-invariant measure of maximal entropy (\emph{Bowen--Margulis measure}) on the unit tangent bundle of \emph{rank-one} (Definition~\ref{def rank-one}) compact convex projective manifolds.

\subsection{Main result}\label{Main result}

Recall that an element of $\PGL(V)$ is said to be \emph{proximal} if it has an attracting fixed point in $\PR(V)$. The \emph{proximal limit set} $\Lambda_\Gamma\subset\PR(V)$ of $\Gamma$ is the closure of the set of attracting fixed points of proximal elements of $\Gamma$; it is $\Gamma$-invariant. We denote by $\Aut(\Omega)$ the group of elements of $\PGL(V)$ preserving~$\Omega$. We introduce the following subset of $T^1M$.

\begin{defi}\label{T1Omegabip}
 Let $\Omega\subset \PR(V)$ be a properly convex open set and $\Gamma\subset\Aut(\Omega)$ a discrete subgroup; denote by $M$ the quotient $\Omega/\Gamma$. The \emph{biproximal unit tangent bundle} of $M$ is
 \[T^1M_{bip}:=\{v\in T^1\Omega \ : \ \phi_{\pm\infty}v\in \Lambda_\Gamma\}/\Gamma\subset T^1M, \]
 where $\phi_{\pm\infty}v=\lim_{t\to\pm\infty}\pi\phi_tv$ are the intersection points of the projective line generated by $v$ with the boundary $\partial\Omega$.
\end{defi}

Note that the biproximal unit tangent bundle is closed and invariant under the action of the geodesic flow. It is contained in the \emph{non-wandering set} $\NW(T^1M,(\phi_t)_{t\in\R})$ (Corollary~\ref{bipsaredense}), which consists of those vectors $v$ whose neighbourhoods contain a geodesic which comes back close to $v$ infinitely often (see Definition~\ref{def_NW}); we write $\NW(T^1M)$ for short when the context is clear. The main result of this paper is the following.

\begin{thm}\label{mixing}
 Let $\Omega\subset \PR(V)$ be a properly convex open set and $\Gamma\subset\Aut(\Omega)$ a discrete subgroup which is strongly irreducible. Set $M=\Omega/\Gamma$, and suppose that $T^1M_{bip}$ is non-empty. Then
 \begin{enumerate}
  \item \label{item:mixing} the geodesic flow $(\phi_t)_t$ on $T^1M_{bip}$ is topologically mixing;
  \item \label{item:max de T1Mbip} if $(\phi_t)_t$ is topologically transitive on an invariant closed subset $A\subset T^1M$ which contains $T^1M_{bip}$, then $A=T^1M_{bip}$.
 \end{enumerate}
\end{thm}

The assumption that $\Gamma$ is strongly irreducible is mild, and one can always restrict to it in the divisible case (see \cite[Th.\,3]{vey} and \cite[\S\,5.1]{benoist_survey}).  When $M$ is compact and $3$-dimensional, we recover Bray's result \cite[Th.\,5.7]{bray_top_mixing} because \cite[Th.\,1.1]{CD4} implies $T^1M_{bip}=T^1M$ in that case. When $\Omega$ is strictly convex, one can see that $T^1M_{bip}=\NW(T^1M)$ (see \cite[\S3.3]{CM2014flot} or Observation~\ref{NW<T1Mcore}), and Crampon--Marquis \cite[Prop.\,6.1]{CM2014flot} showed that in this case the geodesic flow is topologically mixing on $\NW(T^1M)$, if $\partial\Omega$ is smooth (for us smooth means $\mathcal{C}^1$, see Section~\ref{extreme_yet_smooth}). Thus, the point of Theorem~\ref{mixing}.\ref{item:mixing} is to treat the non-strictly convex case, where in general we can only prove that $T^1M_{bip}$ is contained in $\NW(T^1M)$, see Remark~\ref{T1Mbip<NW}.

Another way to formulate Theorem~\ref{mixing}.\ref{item:max de T1Mbip} is to say that $T^1M_{bip}$ is maximal for inclusion among the invariant closed subsets of $T^1M$ on which the geodesic flow is topologically transitive.

Our strategy of proof for Theorem~\ref{mixing}.\ref{item:mixing} is similar to that of Bray in \cite{bray_top_mixing}, but we manage to work without Benoist's geometric description \cite[Th.\,1.1]{CD4} of $3$-dimensional compact convex projective manifolds. Furthermore, we slightly shorten the proofs by using more algebraic arguments inspired by \cite{benoist2000automorphismes,BenoistPropAsymp2} (see Section~\ref{density}): they allow us to prove topological mixing directly, without establishing first topological transitivity and then using a closing lemma and a weak-orbit gluing lemma as in \cite[Th.\,4.4 \& Lem.\,5.3]{bray_top_mixing}.

Another strategy of proof for Theorem~\ref{mixing}.\ref{item:mixing}, in the compact case, could be to find a good geometric description that generalises Benoist's \cite{CD4} to arbitrary dimension. This is an interesting question, and recent work suggests that such a description could exist: Benoist \cite{CD4}, Marquis \cite{LudoThese},  Ballas--Danciger--Lee \cite{BDL_cvxproj_3mfd}, and Choi--Lee--Marquis \cite{choi2016convex} constructed non-strictly convex, compact convex projective manifolds in dimensions $4$ to~$7$ that share a number of nice geometric features with those in dimension $3$; recent work of Bobb \cite{Bobb} extends some results of \cite{CD4} to all dimensions.

\subsection{The biproximal unit tangent bundle}\label{The biproximal unit tangent bundle}

Topological mixing and topological transitivity belong to a family of transitivity properties which have been investigated for geodesic flows on non-positively curved Riemannian manifolds $X$ for many decades, see for instance the classical surveys \cite{hedlund_survey, EHS_survey}. We now briefly relate Theorem~\ref{mixing} to older results for non-positively curved manifolds.

The topological transitivity of $(\phi_t)_{t\in\R}$ on $\NW(T^1M)$ for $M=\Omega/\Gamma$, when $\Omega$ is strictly convex, $\partial\Omega$ is smooth and $\pi_1(M)$ is non-elementary \cite[Prop.\,6.1]{CM2014flot} is analogous to that of $(\phi_t)_{t\in\R}$ on $\NW(T^1X)$ when $X$ is negatively curved and $\pi_1(X)$ is non-elementary, which was proved by Eberlein \cite[Th.\,3.11]{eberlein_geodflowsI}. 

When $\Omega$ is not necessarily strictly convex and $T^1M_{bip}$ is non-empty, the situation is analogous to $X$ being non-positively curved and \emph{rank-one}, \ie having a \emph{rank-one} periodic vector. This notion was introduced by Ballmann--Brin--Eberlein \cite[Def.\,p.\,1]{BBE_rank_one}. By work of Ballmann \cite[Th.\,3.5]{ballmann_axial}, if $X$ is rank-one and $\NW(T^1X)=T^1X$ (\eg if $X$ is rank-one and compact), then $(\phi_t)_{t\in\R}$ is topologically mixing on $T^1X$. Coud\`ene--Schapira studied the action of $(\phi_t)_{t\in\R}$ on $\NW(T^1X)$ without assuming that $\NW(T^1X)=T^1X$; they established \cite[Th.\,5.2]{schapira_generic_measure} the topological transitivity of $(\phi_t)_{t\in\R}$ on some invariant subset $\NW_1(T^1X)$ of $\NW(T^1X)$, defined in \cite[\S5.1]{schapira_generic_measure}, consisting of rank-one vectors with an extra condition.

We wish to interpret $T^1M_{bip}$ as an analogue of $\NW_1(T^1X)$. The definitions of rank-one convex projective manifolds and their rank-one periodic geodesics are now available thanks to the very recent work of M.\,Islam \cite[Def.\,1.3 \& 6.2]{islam_rank_one} and A.\,Zimmer \cite[Def.\,1.1]{zimmer_higher_rank}. Here we adopt Islam's definition, which we reformulate as follows.

\begin{defi}[\!{\!\cite{islam_rank_one}}]\label{def rank-one}
Let $\Omega\subset\PR(V)$ be a properly convex open set, $\Gamma\subset\Aut(\Omega)$ a discrete subgroup, and $M:=\Omega/\Gamma$. A periodic vector $v\in T^1M$ is said to be \emph{rank-one} if for any lift $\tilde{v}\in T^1\Omega$, the points $\phi_{\infty}\tilde{v}$ and $\phi_{-\infty}\tilde v$ are \emph{smooth} (\ie admit a unique supporting hyperplane) and \emph{strongly extremal} (\ie are not contained in any non-trivial segment of the boundary $\partial\Omega$); in this case, any infinite-order element of $\Gamma$ which preserves the orbit of a lift $\tilde{v}$ is said to be \emph{rank-one}. The convex projective manifold (or orbifold) $M$ is \emph{rank-one} if $T^1M$ contains a rank-one periodic vector.
\end{defi}

The classical Fact~\ref{period_is_translation_length} below ensures that rank-one periodic vectors are contained in $T^1M_{bip}$, which is then non-empty whenever $M$ is rank-one. Furthermore, Proposition~\ref{rang un dense dans T1Mbip} tells us that, when they exist, periodic rank-one vectors are dense in $T^1M_{bip}$.

Further evidence for thinking that $T^1M_{bip}$ is analogous to $\NW_1(T^1X)$, is the fact that if $M$ is compact and \emph{higher-rank} (\ie not rank-one), then $T^1M_{bip}$ is empty (see Section~\ref{T1Mbip vide casred} and Remark~\ref{T1Mbip vide cassym}). In this case the proximal limit set is non-empty (Fact~\ref{densitybis}), but any segment between two distinct points of $\Lambda_\Gamma$ is contained in $\partial\Omega$.

To conclude this section, we ask the following question: if $T^1M_{bip}$ is non-empty, is it the whole non-wandering set $\NW(T^1M)$? In the Riemannian setting, if $X$ is compact and rank-one, then $\NW_1(T^1X)$ is dense in $\NW(T^1X)=T^1X$ \cite[\S5]{schapira_generic_measure}. In the paper in preparation \cite{mesureBM}, we prove that this is also true in the convex projective setting: if $M$ is compact and rank-one, then $T^1M_{bip}=\NW(T^1M)=T^1M$.

When the manifold is non-compact, the situation is more subtle. In the Riemannian setting, Coud\`ene--Schapira \cite[\S5.2]{schapira_generic_measure} constructed an example where $X$ is non-compact and $\NW_1(T^1X)$ is non-empty and not dense in $\NW(T^1X)$. In the convex projective setting, $T^1M_{bip}$ may be non-empty and smaller than $\NW(T^1M)$ for non-compact $M$, even when $M$ is convex cocompact in the sense of \cite{fannycvxcocpct}; such an example can be constructed using work in preparation of Danciger--Gu\'eritaud--Kassel \cite{DGKpingpong}, which was advertised in \cite[Prop.\,12.5]{fannycvxcocpct}.

Observe that when $M$ is higher-rank and compact, $\NW(T^1M)$ is different from $T^1M_{bip}$, since, contrary to the latter, the former is non-empty.

\subsection{The higher-rank, irreducible and compact case}

When $M$ is compact, higher-rank and irreducible (in the sense that $\Gamma$ is strongly irreducible), Theorem~\ref{mixing} does not tell us anything since $T^1M_{bip}$ is empty. However, in this case, the investigation of dynamical properties of the geodesic flow happens to be easier, thanks to the recent work of Zimmer \cite[Th.\,1.4]{zimmer_higher_rank}, which classifies these manifolds (this is similar to a classification of compact higher-rank non-positively curved Riemannian manifolds by Ballmann \cite[Cor.\,1]{ballmann_higher_rank} and Burns--Spatzier \cite[Th.\,5.1]{burns_spatzier_higher_rank}). More precisely, he proves that universal covers in $\PR(V)$ of higher-rank irreducible compact convex projective manifolds belong to a narrow and explicit list of properly convex open sets, called \emph{symmetric} (see Section~\ref{symmetric_convex}). We use this to establish the following.

\begin{prop}\label{symmetric_top_mixing}
 Let $M$ be a higher-rank irreducible compact convex projective manifold. Then the non-wandering set of the geodesic flow on $T^1M$ has several (more than one) connected components, and the geodesic flow is topologically mixing on each of them.
\end{prop}

Proposition~\ref{symmetric_top_mixing} is a direct consequence of Proposition~\ref{symmetric_mixing}, where the connected components of the non-wandering set are described more precisely.

\paragraph{Organisation of the paper} In Section~\ref{Reminders} we recall some basic definitions and properties in convex projective geometry. In Section~\ref{Endpoints of biproximal periodic geodesics are smooth} we investigate the regularity of endpoints of biproximal periodic geodesics. In Section~\ref{density} we prove that, when $T^1M_{bip}\neq\emptyset$, the local length spectrum is non-arithmetic; in other words, for every non-empty open subset $U$ of $T^1M_{bip}$, the additive subgroup of $\R$ generated by lengths of biproximal periodic geodesics through $U$ is dense in $\R$. In Section~\ref{Strong stable manifolds} we prove that a geodesic, which has the same endpoint in $\partial\Omega$ as a biproximal periodic geodesic $\gamma$, must in the quotient wrap around closer and closer to $\gamma$ (see Figure~\ref{figure_mixing}). In Section~\ref{Proof of Mixing} we prove Theorem~\ref{mixing} using Sections~\ref{Endpoints of biproximal periodic geodesics are smooth}, \ref{density} and \ref{Strong stable manifolds}, and classical dynamical arguments. In Section~\ref{symmetric_convex} we study the non-wandering set of the geodesic flow on higher-rank irreducible compact convex projective manifolds, and prove Proposition~\ref{symmetric_mixing}. In Appendix~\ref{appendix} we fill in a missing detail in Crampon's original proof of a useful technical lemma in convex projective geometry.

\paragraph*{Acknowledgements}
I am grateful to Yves Benoist, Harrison Bray, Jean-Philippe Burelle, Olivier
Glorieux, Ludovic Marquis, Fr\'ed\'eric Paulin, and Barbara Schapira for helpful discussions and com-
ments. I thank my advisor Fanny Kassel for her time, help, advice and encouragements. This project
received funding from the European Research Council (ERC) under the European Union’s Horizon
2020 research and innovation programme (ERC starting grant DiGGeS, grant agreement No 715982).

\section{Reminders and basic facts}\label{Reminders}

\subsection{Properly convex open subsets of \texorpdfstring{$\PR(\R^{d+1})$}{PRd} and their geodesic flow}\label{distance}

In the whole paper we fix a real vector space $V=\R^{d+1}$. Let $\Omega\subset \PR(V)$ be a properly convex open set. Recall that $\Omega$ admits an $\Aut(\Omega)$-invariant proper metric called the \emph{Hilbert metric} and defined by the following formula: for $(a,x,y,b)\in\partial\Omega\times\Omega\times\Omega\times\partial\Omega$ aligned in this order (see Figure~\ref{figure_distance}),
\[d_\Omega(x,y)=\frac{1}{2}\log([a,x,y,b]),\]
where $[a,x,y,b]$ is the cross-ratio of the four points, normalised so that $[0,1,t,\infty]=t$. 

\begin{figure}
\centering
\begin{tikzpicture}[scale=2]
\coordinate (a) at (-0.7,-0.1);
\coordinate (b) at (1.8,0.1);
\coordinate (c) at (-0.4,-1);
\coordinate (d) at (1.56,-0.6);
\coordinate (e) at (-0.3,0.5);
\coordinate (f) at (1.7,-.3);
\coordinate (g) at ($ (f)!.5!(d) +(.001,0) $);
\coordinate (h) at (-.69,-.5);
\coordinate (i) at ($ (h)!.5!(c) +(-.001,0) $);

\coordinate (z) at (intersection of c--f and h--d);
\coordinate (zp) at ($ (z)!.5!(d) $);
\coordinate (zm) at ($ (c)!.5!(z) $);

\coordinate (x) at ($ (a)!0.2!(b) $);
\coordinate (y) at ($ (a)!0.6!(b) $);
\coordinate (v) at ($ (x)!0.3!(y) $);
\coordinate (phitv) at ($ (y)!.3!(b) $);

\length{e,a,h,i,c,d,g,f,b}
\cvx{e,a,h,i,c,d,g,f,b}{1}

\draw (a)--(b);

\draw (a) node{$\bullet$} node[left]{$a$};
\draw (x) node{$\bullet$} node[above]{$x$};
\draw (y) node{$\bullet$} node[above]{$y$};
\draw (b) node{$\bullet$} node[right]{$b$};
\draw (c) node[green!60!black]{$\bullet$};
\draw (d) node[green!60!black]{$\bullet$};
\draw (f) node[green!60!black]{$\bullet$};
\draw (h) node[green!60!black]{$\bullet$};
\draw (z) node[green!60!black]{$\bullet$};
\draw (zp) node[green!60!black]{$\bullet$}node[green!60!black,below]{$y'$};
\draw (zm) node[green!60!black]{$\bullet$}node[green!60!black,below]{$x'$};
\draw (e) node[above left]{$\Omega$};
\draw [->, thick] (x) -- (v) node[below]{$v$};
\draw [->, thick] (y) -- (phitv) node[below]{$\phi_tv$};
\draw [green!60!black,dashed,very thin] (c)--(f);
\draw [green!60!black,dashed,very thin] (h)--(d);
\draw [green!60!black] (zm)--(z);
\draw [green!60!black] (z)--(zp);
\draw [green] (d)--(f);
\draw [green] (h)--(c);

\end{tikzpicture}
\caption{The Hilbert metric and the geodesic flow ($t=d_\Omega(x,y)$)}\label{figure_distance}
\end{figure}
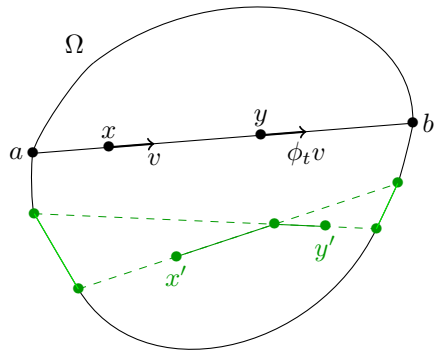

Recall that if $\Omega$ is an ellipsoid, then $(\Omega,d_\Omega)$ is the Klein model of the real hyperbolic space of dimension $d$, and if $\Omega$ is a $d$-simplex, then $(\Omega,d_\Omega)$ is isometric to $\R^d$ endowed with a hexagonal norm.

Any discrete subgroup $\Gamma\subset\PGL(V)$ of automorphisms of $\Omega$ preserves $d_\Omega$, hence must act properly discontinuously on $\Omega$ and therefore the quotient $M=\Omega/\Gamma$ is an orbifold. Furthermore, $M$ is a manifold if the action is free (\ie if $\Gamma$ is torsion-free, by Brouwer's fixed point theorem, applied to the convex hull of a finite orbit of a torsion element). Note that by Selberg's lemma \cite{selberg_lemma}, if $\Gamma$ is finitely generated, then it admits a torsion-free finite-index subgroup. We will work in general with $\Gamma$ not necessarily torsion-free, so we set the notation $T^1M=T^1\Omega/\Gamma$. 

The intersections of $\Omega$ with projective lines can be parametrised to be geodesics, which are said to be \emph{straight}. However, an interesting feature in the non-strictly convex case is that when there are two coplanar non-trivial segments in the boundary $\partial\Omega$, one can construct geodesics which are not straight, see for instance the broken green segment between $x'$ and $y'$ in Figure~\ref{figure_distance}. In order to define the geodesic flow we only take into account straight geodesics: for $v$ in $T^1\Omega$, let $t\mapsto c(t)$ be the parametrisation of the projective line generated by $v$ such that $c$ is an isometric embedding from $\R$ to $\Omega$ and $c'(0)=v$. For $t\in\R$ we set $\phi_t(v)=c'(t)\in T^1\Omega$. See Figure~\ref{figure_distance}.

The geodesic flow on $T^1M=T^1\Omega/\Gamma$ is well defined because the two actions of $\Aut(\Omega)$ and $(\phi_t)_{t\in\R}$ on $T^1\Omega$ commute. We denote by $\pi :T^1M\rightarrow M$ and $\pi : T^1\Omega\rightarrow \Omega$ the natural projections, and we consider the following metric on $T^1\Omega$:
\begin{align}\label{eq:dT1Om}
 \forall v,w\in T^1\Omega, \quad d_{T^1\Omega}(v,w) & =  \max_{0\leq t\leq 1}d_\Omega(\pi\phi_tv,\pi\phi_tw).
\end{align}

The following remark is a direct consequence of the definition of the Hilbert metric.

\begin{rqq}\label{Hilbert VS Euclidean}
Let $\Omega\subset\PR(V)$ be a properly convex open set, and fix an affine chart containing $\overline{\Omega}$. Then 
\[\overline{B}_\Omega(x,r) \subset (1-e^{-2r})(\overline{\Omega}-x) +x\]
for all $x\in\Omega$ and $r>0$, where $\overline{B}_\Omega(x,r)$ is the closed ball of radius $r$, centred at $x$, for the metric $d_\Omega$, and $(1-e^{-2r})(\overline{\Omega}-x) +x$ is the image of $\overline{\Omega}$ under the homothety (of the affine chart) centred at $x$ and with ratio $1-e^{-2r}$.
\end{rqq}

\subsection{Smooth and extremal points of the boundary}\label{extreme_yet_smooth}

 We recall here some terminology on convex sets. Let $\Omega\subset \PR(V)$ be a properly convex open set. Let $\xi\in\partial\Omega$ be a point of the boundary.
 \begin{itemize}
  \item A \emph{supporting hyperplane} of $\Omega$ at $\xi$ is a hyperplane which contains $\xi$ and does not intersect $\Omega$. Note that there always exists such a hyperplane.
  \item As in Definition~\ref{def rank-one}, we shall say that $\xi$ is a \emph{smooth} point of $\partial\Omega$ (this is commonly also called a $\mathcal{C}^1$ point) if there is only one supporting hyperplane of $\Omega$ at $\xi$, which we then denote by $T_\xi\partial\Omega$.
  \item The point $\xi$ is said to be \emph{extremal} if it is not contained in the relative interior of a non-trivial segment contained in the boundary $\partial\Omega$.
  \item Observe that $\Omega$ is \emph{strictly convex} if and only if all points of $\partial\Omega$ are extremal.
  \item As in Definition~\ref{def rank-one}, we shall say that $\xi$ is \emph{strongly extremal} if it is not contained in any non-trivial segment contained in the boundary $\partial\Omega$. Observe that the endpoint of a segment contained in $\partial\Omega$ may be extremal, but is never strongly extremal.
 \end{itemize}

\subsection{Proximal linear transformations}\label{Proximal linear transformations}

In this section we recall the notion of a proximal linear transformation, which was used in the definition of the proximal limit set $\Lambda_\Gamma$ and the biproximal unit tangent bundle $T^1M_{bip}$ in Section~\ref{Main result}.

\begin{nota}\label{oplus}
If $W_1$ and $W_2$ are two subspaces of $V$ such that $W_1\cap W_2=\{0\}$, we write $W_1\oplus W_2\subset V$ for their direct sum and $\PR(W_1)\oplus\PR(W_2)=\PR(W_1\oplus W_2)$ for its projectivisation. In particular, if $x,y\in \PR(V)$ are two distinct points, we write $x\oplus y$ for the projective line through $x$ and $y$.
\end{nota}

\begin{defi}\label{def_prox}
A linear transformation $g\in \End(V)$ is \emph{proximal} if it has exactly one complex eigenvalue with maximal modulus among all eigenvalues, and if this eigenvalue has multiplicity $1$. The associated eigenline in $\PR(V)$ is the attracting fixed point of $g$ and is denoted by $x_g^+$.

An invertible linear transformation $g\in\GL(V)$ is said to be \emph{biproximal} if $g$ and $g^{-1}$ are both proximal. The attracting fixed point of $g^{-1}$ is the repelling fixed point of $g$ and is denoted by $x_g^-$. The projective line $x_g^+\oplus x_g^-$ (see Notation~\ref{oplus}) is the \emph{axis} of $g$ and is denoted by $\axis(g)$. The $g$-invariant complementary subspace to the axis of $g$ is denoted by $x_g^0$. Note that the notions of biproximality, attracting/repelling fixed point, and axis, are well defined for the image of $g$ in $\PGL(V)$.
\end{defi}

\begin{rqq}\label{prox_are_open}
The set of proximal linear transformations is open in $\End(V)$, and the map sending a proximal linear transformation to the pair (attracting fixed point, associated eigenvalue) is continuous.
\end{rqq}

\begin{rqq}\label{minimality}
 As observed by Benoist \cite[Lem.\,3.6.ii]{BenoistPropAsymp}, for any subgroup $\Gamma\subset\PGL(V)$ which is irreducible (\ie preserves no proper subspace of $\PR(V)$) and contains a proximal element, the proximal limit set is the smallest closed $\Gamma$-invariant non-empty subset of $\PR(V)$; in particular, the action of $\Gamma$ on  $\Lambda_\Gamma$ is minimal (\ie any orbit is dense). Indeed, consider any proximal element $\gamma\in\Gamma$, and let $\PR(W)\subset \PR(V)$ be the $\gamma$-invariant complementary subspace  to $x_\gamma^+$. By irreducibility, any closed $\Gamma$-invariant non-empty subset $X\subset\PR(V)$ contains a point $x$ outside $\PR(W)$, and then $x_\gamma^+$, which is the limit of the sequence $(\gamma^nx)_{n\in\N}$, belongs to $X$.
\end{rqq}

\subsection{Periodic geodesics and automorphisms of \texorpdfstring{$\Gamma$}{Gamma}}\label{Periodic geodesics and automorphisms of Gamma}

In this section we recall the link between periodic geodesics in $T^1\Omega/\Gamma$ and conjugacy classes of $\Gamma$. Let $\Omega\subset \PR(V)$ be a properly convex open set. For $g\in \GL(V)$, we denote by $\lambda_1(g)\geq\dots\geq\lambda_{d+1}(g)$ the non-increasing sequence of logarithms of moduli of eigenvalues of $g$; we set

\begin{equation}\label{l(g)}
\ell(g):=\frac{1}{2}(\lambda_1(g)-\lambda_{d+1}(g)).
\end{equation}
Observe that $\ell(g)$ only depends on the class of $g$ in $\PGL(V)$. If $g$ preserves $\Omega$, then
\begin{equation}\label{longueur de translation}
\ell(g)=\inf\{d_\Omega(x,g\cdot x): \ x\in\Omega\}\geq 0.
\end{equation}
The right-hand side of \eqref{longueur de translation} is called the \emph{translation length} of $g$. See \cite[Prop.\,2.1]{CLT2015cvxisom} for a proof.

Combined with an elementary computation, \eqref{longueur de translation} yields:

\begin{fait}\label{period_is_translation_length} Let $\Omega\subset\PR(V)$ be a properly convex open set, let $\Gamma\subset\Aut(\Omega)$ be a discrete subgroup, and let $M=\Omega/\Gamma$. Then for any infinite geodesic $\tilde{c}$ in $\Omega$ that lifts a periodic straight geodesic $c$ of $M$, there is an automorphism $\gamma\in\Gamma$ which preserves it and acts by translation on it. Let $\tilde{\gamma}\in\GL(V)$ be a lift of $\gamma$. The endpoints in $\partial\Omega$ of $\tilde{c}$ are fixed by $\gamma$, the associated eigenvalues of $\tilde{\gamma}$ are $\exp(\lambda_1(\tilde{\gamma}))$ and $\exp(\lambda_{d+1}(\tilde{\gamma}))$, and the length of $c$ is the translation length of $\gamma$. If furthermore these endpoints are extremal, then $\gamma$ is biproximal.
\end{fait}

By Fact~\ref{period_is_translation_length}, rank-one elements of $\Gamma$ (Definition~\ref{def rank-one}) are biproximal. Hence rank-one periodic vectors of $T^1M$ belong to the biproximal unit tangent bundle $T^1M_{bip}$ (Definition~\ref{T1Omegabip}).

\begin{defi}\label{biprox_orbit}
 Let $\Omega\subset \PR(V)$ be a properly convex open set and $\Gamma\subset\Aut(\Omega)$ a discrete subgroup. Let $\gamma\in\Gamma$ be a biproximal element whose axis meets $\Omega$. Then the periodic geodesic associated to $\gamma$ is said to be \emph{biproximal}, and the unit tangent vectors along this geodesic are said to be \emph{biproximal} periodic.
\end{defi}

There are cases where $\gamma\in\Gamma$ is biproximal but its axis does not intersect $\Omega$ (\eg when $\Omega$ is a triangle, or is symmetric as in Section~\ref{symmetric_convex}). Then we cannot make sense of a straight periodic geodesic associated to~$\gamma$.

\subsection{Density of biproximal geodesics}

We gather here two results of Benoist which imply that biproximal periodic vectors are dense in $T^1M_{bip}$.

\begin{fait}[\!{\!\cite[Prop.\,1.1]{benoist2000automorphismes}} \& {\cite[Lem.\,3.6.iv]{BenoistPropAsymp}}]\label{densitybis}
 Let $\Gamma\subset \PGL(V)$ be a strongly irreducible subgroup.
\begin{enumerate}
\item \label{Item : proximalite} If $\Gamma$ preserves a properly convex open set $\Omega\subset\PR(V)$, then it contains a proximal element.
\item \label{Item : densite bip} If $\Gamma$ contains a proximal element, then the following subset is dense in $\Lambda_\Gamma\times\Lambda_\Gamma$:
 \[\{(x_\gamma^+,x_\gamma^-)\in \Lambda_\Gamma\times\Lambda_\Gamma: \gamma\in \Gamma \text{ biproximal}\}.\]
\end{enumerate}
\end{fait}

\begin{cor}\label{bipsaredense}
 Let $\Omega\subset \PR(V)$ be a properly convex open set and $\Gamma\subset\Aut(\Omega)$ a strongly irreducible discrete subgroup. Denote by $M$ the quotient $\Omega/\Gamma$, and suppose that $T^1M_{bip}$ is non-empty. Then biproximal periodic geodesics exist and their images in $T^1M$ are dense in $T^1M_{bip}$.
\end{cor}

\subsection{The non-wandering set}\label{The non-wandering set}

In this section we recall the definition of the non-wandering set and the link between the non-wandering set of the geodesic flow on $T^1M$ and the non-wandering set of the actions of $\Gamma$ and $\Aut(\Omega)$ on the space of geodesics of $\Omega$. This will be used in Section~\ref{symmetric_convex}.

\begin{defi}\label{def_NW}
 Let $X$ be a locally compact topological space equipped with a continuous action by a locally compact group $G$. The \emph{non-wandering set} $\NW(X,G)$ is the set of points $x$ in $X$ such that for any compact neighbourhood $U$ of $x$, the set $\{g\in G: gU\cap U\neq\emptyset\}$ is non-compact.
\end{defi}

In other words, it is the set of points all of whose neighbourhoods come back infinitely often under the action; we call such points \emph{non-wandering}. The non-wandering set is closed and $G$-invariant. Note that if $X$ is compact but $G$ is not, then the non-wandering set is non-empty. When $G$ is $\R$, \ie when we have a flow $(\phi_t)_{t\in\R}$ on $X$, observe that given a non-wandering point $x\in X$ and a neighbourhood $U$ of $x$, one can find arbitrarily large \emph{positive} times $t$ such that $\phi_t U\cap U\neq\emptyset$; indeed, for any $t\in\R$, if $\phi_t U\cap U\neq\emptyset$, then $\phi_{-t} U\cap U\neq\emptyset$.

In our setting, there are three non-wandering sets of interest for us. Let $\Omega\subset\PR(V)$ be a properly convex open set, let $\Gamma\subset\Aut(\Omega)$ be a closed subgroup of automorphisms, and let $M$ be the quotient $\Omega/\Gamma$. One can first consider the non-wandering set $\NW(T^1M,(\phi_t)_{t\in\R})$ of the geodesic flow.

\begin{rqq}\label{T1Mbip<NW}
Any vector of $T^1M$ which is tangent to a periodic straight geodesic belongs to the non-wandering set $\NW(T^1M,(\phi_t)_{t\in\R})$. As a consequence, if $\Gamma$ is strongly irreducible, then $T^1M_{bip}$ is contained in $\NW(T^1M,(\phi_t)_{t\in\R})$ by Corollary~\ref{bipsaredense}.
\end{rqq}

 Let us denote by $\Geod(\Omega)=T^1\Omega/(\phi_t)_{t\in\R}$ the set of straight geodesics of $\Omega$: it is an open subset of $\partial\Omega^2$, consisting of the pairs $(x,y)$ such that $x\neq y$ and the projective line through $x$ and $y$ meets $\Omega$. The group $\Gamma$ naturally acts on $\Geod(\Omega)$ and one can consider its non-wandering set $\NW(\Geod(\Omega),\Gamma)$. Finally, one can consider the two commutative and proper actions of $\Gamma$ and 
$\R$ (by the geodesic flow) on $T^1\Omega$, it yields the non-wandering set $\NW(T^1\Omega,\Gamma\times\R)$. All three of these non-wandering sets are actually identified in the following sense. Denote the canonical projections by $\pi_\R: T^1\Omega\rightarrow\Geod(\Omega)$ and $\pi_\Gamma:T^1\Omega\rightarrow T^1 M$. Then
\[\pi_\R^{-1}(\NW(\Geod(\Omega),\Gamma))=\NW(T^1\Omega,\Gamma\times\R)=\pi_\Gamma^{-1}(\NW(T^1M,(\phi_t)_{t\in\R})).\]
We will use this while studying symmetric properly convex open sets in Section~\ref{symmetric_convex}.

To end this section, we observe that the non-wandering set $\NW(T^1M,(\phi_t)_{t\in\R})$ is contained in another $(\phi_t)_{t\in\R}$-invariant subset $T^1M$, defined similarly to $T^1M_{bip}$, but using another limit set in the boundary. Recall that Danciger, Gu\'eritaud, and Kassel \cite[Def.\,1.10]{fannycvxcocpct} defined the \emph{full orbital limit set} $\Lambda_\Gamma^{orb}\subset\partial\Omega$ as the union, over all $x\in\Omega$, of the set of accumulation points of the orbit $\Gamma\cdot x$; the full orbital limit set always contains the proximal limit set. Similarly to $T^1M_{bip}$, we can consider 
\[T^1M_{core}:=\{v\in T^1\Omega : \phi_{\pm\infty}v\in\Lambda_\Gamma^{orb}\}/\Gamma\subset T^1M.\]

\begin{obs}\label{NW<T1Mcore}
 Let $\Omega\subset\PR(V)$ be a properly convex open set, let $\Gamma$ be a discrete group of automorphisms of $\Omega$, and denote by $M$ the quotient $\Omega/\Gamma$. Then
 \[\NW(T^1M,(\phi_t)_{t\in\R})\subset T^1M_{core}.\]
\end{obs}

\begin{proof}
 Consider a vector $v\in T^1\Omega$ whose projection in $T^1M$ is non-wandering. Let $x=\pi v$ be the footpoint of $v$. We want to show that $\phi_{\infty}v$ is an accumulation point of $\Gamma\cdot x$. Since the projection of $v$ in $T^1M$ is non-wandering, we can find sequences of vectors $(v_n)_n$ in $T^1\Omega$ converging to $v$, of positive times $(t_n)_n$ going to infinity, and of automorphisms $(\gamma_n)_n$ in $\Gamma$ such that $(d_{T^1\Omega}(\phi_{t_n}v_n,\gamma_nv))_n$ tends to zero. Since $(v_n)_n$ tends to $v$ and $(t_n)_n$ goes to infinity, $(\pi\phi_{t_n}v_n)_n$ must converge to $\phi_\infty v$. By Remark~\ref{Hilbert VS Euclidean}, the fact that $(d_\Omega(\pi\phi_{t_n}v_n,\gamma_nx))_n$ tends to zero implies that $(\gamma_nx)_n$ also converges to $\phi_\infty v$ in $\PR(V)$.
\end{proof}

The full orbital limit set $\Lambda_\Gamma^{orb}$ and Observation~\ref{NW<T1Mcore} will not be used in the remainder of the paper.

\subsection{The biproximal unit tangent bundle of reducible compact convex projective manifolds}\label{T1Mbip vide casred}

In this section we explain, for completeness, why the biproximal unit tangent bundle of a reducible compact convex projective manifold is empty. This is not needed anywhere in the paper.

Vey \cite[Th.\,3]{vey} (see \cite[\S\,5.1]{benoist_survey}) proved that, given any properly convex open set $\Omega\subset\PR(V)$ divided by a discrete group $\Gamma\subset\Aut(\Omega)$, the group $\Gamma$ is \emph{not} strongly irreducible if and only if $\Omega$ is \emph{decomposable}, \ie there exists a decomposition $V=V_1\oplus V_2$ and convex open cones $\Ccal_i\subset V_i$ with $\PR(\Ccal_i)\subset\PR(V_i)$ properly convex for $i=1,2$, such that $\Omega=\PR(\Ccal_1+\Ccal_2)$; in this case we say that $M=\Omega/\Gamma$ is a \emph{reducible} compact convex projective manifold.

\begin{lemma}
 Suppose that $\dim(V)=d+1>2$. Let $\Omega\subset\PR(V)$ be a decomposable properly convex open set, and $\Gamma\subset\Aut(\Omega)$ a discrete subgroup. Then the quotient $M=\Omega/\Gamma$ is higher-rank and $T^1M_{bip}$ is empty. In particular, reducible compact convex projective manifolds are higher-rank and have an empty biproximal unit tangent bundle.
\end{lemma}

\begin{proof}
 Let us show that $\partial\Omega$ contains no strongly extremal point (this implies that $M$ is higher-rank), and that $[x,y]\subset\partial\Omega$ for all extremal points of $\partial\Omega$ (this implies that $T^1M_{bip}$ is empty since one can check that $\Lambda_{\Gamma}$ is contained in the closure of the set of extremal points). Consider a decomposition $\Omega=\PR(\Ccal_1+\Ccal_2)$, where $\Ccal_i\subset V_i$ is a convex cone for $i=1,2$, with $V=V_1\oplus V_2$.
 
 Observe that the boundary of $\Omega$ is equal to $\PR(\partial\Ccal_1+\overline\Ccal_2)\cup\PR(\overline\Ccal_1+ \partial\Ccal_2)\cup \PR(\overline\Ccal_1+ 0)\cup\PR(0+\overline\Ccal_2)$. Take $x\in\partial\Omega$, and let us check that $x$ is not strongly extremal. If $x=[tv+(1-t)w]$ for $v\in\partial\Ccal_i$ and $w\in\overline\Ccal_j$ and $0<t\leq 1$, with $i\neq j\in\{1,2\}$, then $\{[sv+(1-s)w:0\leq s\leq 1\}$ is a non-trivial segment of $\partial\Omega$ that contains $x$. If $x\in\PR(\overline\Ccal_i)$, then we take $y\in\PR(\partial\Ccal_i)\smallsetminus\{x\}$, or in $\PR(\partial\Ccal_j)$, with $i\neq j\in\{1,2\}$ ($y$ exists because $\dim(V)>2$), and $[x,y]\subset\partial\Omega$ is a non-trivial segment.

 Observe that the set of extremal points of $\Omega$ is equal to $\PR(\widetilde E_1+ 0)\cup\PR(0+ \widetilde E_2)$, where $\widetilde E_i$ is the preimage in $V$ of the set of extremal points of $\PR(\Ccal_i)$ for $i=1,2$. If $x,y\in\PR(\widetilde E_i)$ with $i=1,2$, then $[x,y]\subset\PR(\overline\Ccal_i)\subset\partial\Omega$. If $x\in\PR(\widetilde E_1)$ and $y\in\PR(\widetilde E_2)$, then either $x\in\PR(\partial\Ccal_1)$ or $y\in\PR(\partial\Ccal_2)$ (since $\dim(V)>2$), and $[x,y]\subset\PR(\partial\Ccal_1+\overline\Ccal_2)\cup\PR(\overline\Ccal_1+\partial\Ccal_2)\subset\partial\Omega$.
\end{proof}

\section{Endpoints of biproximal periodic geodesics are smooth}\label{Endpoints of biproximal periodic geodesics are smooth}

This section contains an elementary result (Lemma~\ref{biproxiattractingpointsaresmooth}) which will be used in the proof of topological mixing in Section~\ref{Proof of Mixing}. Furthermore, various characterisations of the rank-one property, for instance in terms of duality, are given in Lemma~\ref{equivalences rang un}. Finally, we use Lemma~\ref{biproxiattractingpointsaresmooth} to justify a claim of the introduction: namely, that rank-one periodic geodesics are dense in the biproximal unit tangent bundle of rank-one manifolds (Proposition~\ref{rang un dense dans T1Mbip}).

\subsection{On the regularity of endpoints of biproximal periodic geodesics}

The main consequence of the following lemma is that we will be able to apply Proposition~\ref{strongstablemanifolds}.\ref{Item : variete stable} to biproximal periodic vectors.

\begin{lemma}\label{biproxiattractingpointsaresmooth}
Let $\Omega\subset \PR(V)$ be a properly convex open set and $g\in\Aut(\Omega)$ a biproximal element. Then $\axis(g)=x_g^-\oplus x_g^+$ intersects $\Omega$ if and only if $x_g^+$ is smooth; in this case $T_{x_g^+}\partial\Omega=x_g^+\oplus x_g^0$.
\end{lemma}

\begin{proof}
Assume that $x_g^+$ is not smooth. Then there is a supporting hyperplane $H$ of $\Omega$ at $x_g^+$ which is different from $x_g^+\oplus x_g^0$. Let $x\in H\smallsetminus (x_g^+\oplus x_g^0)$, so that the sequence $(g^{-n}x)_n$ tends to $x_g^-$. The sequence of projective lines through $g^{-n}x$ and $x_g^+$ must converge to $\axis(g)$, and they are all contained in $\PR(V)\smallsetminus\Omega$ which is closed. Therefore $\PR(V)\smallsetminus\Omega$ must contain $\axis(g)$ as well. 

Conversely if $\PR(V)\smallsetminus\Omega$ contains $\axis(g)$, then $x_g^+$ has a supporting hyperplane which contains $\axis(g)$, and which is therefore different from the supporting hyperplane $x_g^0\oplus x_g^+$.
\end{proof}

\subsection{Rank-one periodic geodesics and their dual}

In this section we give several equivalent conditions for an automorphism of a properly convex open set to be rank-one (Definition~\ref{def rank-one}), which follow from Lemma~\ref{biproxiattractingpointsaresmooth}. This will be used in Section~\ref{Density of rank-one periodic geodesics}, and may be interesting in its own right.

Let us recall the notion of duality for properly convex open sets. We identify the dual projective space $\PR(V^*)$ with the set of projective hyperplanes of $\PR(V)$. Let $\Omega$ be a properly convex open subset of $\PR(V)$. The dual of $\Omega$, denoted by $\Omega^*$, is the properly convex open subset of $\PR(V^*)$ defined as the set of projective hyperplanes which do not intersect $\overline{\Omega}$. We naturally identify $\PGL(V)$ and $\PGL(V^*)$, then $\Aut(\Omega)$ identifies with $\Aut(\Omega^*)$, and the attracting (\resp repelling) fixed point of the action on $\PR(V^*)$ of any biproximal element $g\in\PGL(V)$ is $x_g^+\oplus x_g^0$ (\resp $x_g^-\oplus x_g^0$).

\begin{lemma}\label{equivalences rang un}
Let $\Omega\subset \PR(V)$ be a properly convex open set and $g\in\Aut(\Omega)$ a biproximal element. Then the following are equivalent: 
\begin{enumerate}[label=(\alph*)]
\item \label{Item : a} $g$ is rank-one;
\item \label{Item : b} $x_g^+,x_g^-\in\partial\Omega$ are smooth and strongly extremal points;
\item \label{Item : c} $x_g^+$ is strongly extremal;
\item \label{Item : d} $g$ seen as an automorphism of $\Omega^*$ is rank-one;
\item \label{Item : e} the axis of $g$ in $\PR(V)$ intersects $\Omega$, and the axis of $g$ in $\PR(V^*)$ intersects $\Omega^*$.
\end{enumerate}
\end{lemma}

We will need in the proof an elementary fact concerning duality of properly convex open sets, whose proof is left to the reader. Recall that the canonical isomorphism between $V$ and $V^{**}$ identifies $\Omega$ with $\Omega^{**}$. By definition of $\Omega^*$, the boundary $\partial\Omega^*$ is the set of supporting hyperplanes of $\Omega$; by duality $\partial\Omega=\partial\Omega^{**}$ is the set of supporting hyperplanes of $\Omega^*$.

\begin{fait}\label{dualité}
Let $\Omega\subset\PR(V)$ be a properly convex open set.
\begin{enumerate}[label=(\roman*)]
\item \label{Item : dualite 1} A smooth point $x\in\partial\Omega$ is strongly extremal if and only if the tangent space $T_x\partial\Omega$ is a smooth point of $\partial\Omega^*$; in this case $T_x\partial\Omega$ is strongly extremal.
\item \label{Item : dualite 2} For any $H,H'\in\partial\Omega^*$, the segment $[H,H']\subset \overline{\Omega}^*$ is contained in $\partial\Omega^*$ if and only if $H\cap H'\cap\partial\Omega$ is non-empty.
\end{enumerate}
\end{fait}

\begin{proof}[Proof of Lemma~\ref{equivalences rang un}]
The fact that \ref{Item : a}, \ref{Item : b} and \ref{Item : c} are equivalent is a particular case of \cite[Prop.\,6.3]{islam_rank_one}; for clarity we give a complete proof of Lemma~\ref{equivalences rang un}.
\begin{itemize}
\item \ref{Item : b} implies \ref{Item : a} by definition, and the converse holds by Fact~\ref{period_is_translation_length}.
\item \ref{Item : b} and \ref{Item : d} are equivalent by Fact~\ref{dualité}.\ref{Item : dualite 1} and the fact that \ref{Item : a} and \ref{Item : b} are equivalent.
\item That \ref{Item : b} implies \ref{Item : c} is immediate.
\item Let us prove that \ref{Item : c} implies \ref{Item : e}. Assume that $x_g^+$ is strongly extremal. Then $[x_g^+,x_g^-]$ is not contained in $\partial\Omega$, so the axis of $g$ in $\PR(V)$ intersects $\Omega$. Furthermore, $x_g^0\cap\partial\Omega$ is contained in $(x_g^0\oplus x_g^+)\cap\partial\Omega\smallsetminus\{x_g^+\}$ which is empty. By Fact~\ref{dualité}.\ref{Item : dualite 2} this exactly means that the axis of $g$ in $\PR(V^*)$ intersects $\Omega^*$.
\item Let us prove that \ref{Item : e} implies \ref{Item : b}. Assume that the axis of $g$ in $\PR(V)$ intersects $\Omega$, and that the axis of $g$ in $\PR(V^*)$ intersects $\Omega^*$. By Lemma~\ref{biproxiattractingpointsaresmooth}, the points $x_g^+,x_g^-\in\partial\Omega$ and $(x_g^+\oplus x_g^0)$, $(x_g^-\oplus x_g^0)\in\partial\Omega^*$ are smooth. By Fact~\ref{dualité}.\ref{Item : dualite 1}, the points $x_g^+$ and $x_g^-$ are strongly extremal.\qedhere
\end{itemize}
\end{proof}

\subsection{Density of rank-one periodic geodesics}\label{Density of rank-one periodic geodesics}

The following proposition justifies a claim of the introduction. It will not be used in the remainder of this paper.

\begin{prop}\label{rang un dense dans T1Mbip}
 Let $\Omega\subset \PR(V)$ be a properly convex open set. Then 
\begin{enumerate}
\item \label{Item : rg1 dense} for any pair $(\xi,\eta)\in\Geod(\Omega)$ such that $\xi$ is strongly extremal, there exist neighbourhoods $U$ of $\xi$ and $V$ of $\eta$ in $\PR(V)$ such that for any biproximal automorphism $g\in\Aut(\Omega)$, if $(x_g^+,x_g^-)\in U\times V$, then $g$ is rank-one;
 \item if $\Gamma\subset\Aut(\Omega)$ is a strongly irreducible discrete subgroup such that $M=\Omega/\Gamma$ is rank-one, then rank-one periodic geodesics are dense in $T^1M_{bip}$. In particular $T^1M_{bip}$ is not empty.
\end{enumerate}
\end{prop}

\begin{proof}
\begin{enumerate}
\item Let us assume by contradiction that there is a sequence of biproximal automorphisms $(g_n)_{n\in\N}$ which are not rank-one, and such that $(x_{g_n}^+)_{n\in\N}$ and $(x_{g_n}^-)_{n\in\N}$ respectively converge to $\xi$ and $\eta$. For $n$ large enough, $(x_{g_n}^-,x_{g_n}^+)\in\Geod(\Omega)$, hence by Lemma~\ref{equivalences rang un}.\ref{Item : e}, the axis of $g_n$ in $\PR(V^*)$ does not intersect $\Omega^*$, which exactly means that there exists $\xi_n\in x_{g_n}^0\cap\partial\Omega$, by Fact~\ref{dualité}.\ref{Item : dualite 2}. Up to extraction we can assume that $(\xi_n)_{n\in\N}$ converges to some $\xi'\in\partial\Omega$. Since 
\[[\xi_n,x_{g_n}^+]\subset (x_{g_n}^0\oplus x_{g_n}^+)\cap\overline{\Omega}\subset\partial\Omega\]
for all $n$, passing to the limit we obtain that $[\xi',\xi]\subset\Omega$, which implies that $\xi'=\xi$, because $\xi$ is strongly extremal. Similarly, $[\xi_n,x_{g_n}^-]\subset\partial\Omega$ for all $n$, so $[\xi,\eta]\subset\partial\Omega$, which is a contradiction.
\item We denote by $\partial_{sse}\Omega$ the set of smooth and extremal points of $\partial\Omega$. By assumption, the $\Gamma$-invariant subset $\Lambda_\Gamma\cap\partial_{sse}\Omega$ is non-empty, hence dense in $\Lambda_\Gamma$ since the action of $\Gamma$ on $\Lambda_\Gamma$ is minimal (Remark~\ref{minimality}). Therefore, it is enough to show that attracting/repelling pairs of rank-one biproximal elements of $\Gamma$ are dense in $(\Lambda_\Gamma\cap\partial_{sse}\Omega)^2$. Consider a pair of points $\xi\neq\eta$ in $\Lambda_\Gamma\cap\partial_{sse}\Omega$, and a neighbourhood $U$ of $(\xi,\eta)$ in $\partial\Omega^2$; let us check that $U$ contains the attracting/repelling pair of a rank-one element of $\Gamma$. By Proposition~\ref{rang un dense dans T1Mbip}.\ref{Item : rg1 dense}, there exist a sub-neighbourhood of $U$ which does not contain the attracting/repelling pair of a non-rank-one biproximal element, and by Fact~\ref{densitybis}.\ref{Item : densite bip}, this sub-neighbourhood must contain the attracting/repelling pair of a biproximal (hence rank-one) element.\qedhere
\end{enumerate}
\end{proof}

\section{Non-arithmeticity of the length spectrum}\label{density}

In this section we prove the following.
 
\begin{prop}\label{local_non_arithmeticity}
 Let $\Omega\subset \PR(V)$ be a properly convex open set and $\Gamma\subset\Aut(\Omega)$ a strongly irreducible discrete subgroup. Denote by $M$ the quotient $\Omega/\Gamma$. Let $U\subset T^1M_{bip}$ be a non-empty open set. Then the additive group generated by the lengths of biproximal periodic geodesics \emph{through $U$} is dense in $\R$.
\end{prop}

In other words, not only are biproximal periodic geodesics dense in $T^1M_{bip}$, but moreover the \emph{local} length spectrum (through any open set $U$) is non-arithmetic. The idea that the topological mixing of the geodesic flow should be equivalent to the non-arithmeticity of the length spectrum is not new; this question was elucidated in the context of negatively curved Riemannian manifolds by Dal'bo \cite[Th.\,A]{Dal00}.

Our proofs are heavily influenced by the work of Benoist \cite{benoist2000automorphismes, BenoistPropAsymp2}; see also \cite[Ch.\,7]{benoist2016random}.

\subsection{Density of the group generated by Jordan projections}

We gather here two results which imply that any strongly irreducible semi-group of automorphisms of a properly convex open set has a non-arithmetic length spectrum.

A proof of the following result can be found in \cite[Prop.\,6.5]{CM2014flot}.

\begin{fait}[\!{\!\!\cite[Rem.\,p.\,17]{benoist2000automorphismes}}]\label{zclosuresemisimple}
 Let $\Omega\subset \PR(V)$ be a properly convex open set and $\Gamma\subset\SL(V)$ a strongly irreducible discrete subgroup preserving $\Omega$. Then the Zariski closure of $\Gamma$ in $\SL(V)$ is semi-simple and non-compact.
\end{fait}

The following fact uses the language of semi-simple Lie groups, see \eg \cite[Ch.\,6]{benoist2016random} for definitions. 

\begin{fait}[\!{\!\cite[Prop.\,p.2]{BenoistPropAsymp2}}]\label{PropAsymp2}
 Let $G$ be a connected real semi-simple linear Lie group. Let $\mathfrak{a}_G$ be a Cartan subspace of its Lie algebra, let $\mathfrak{a}_G^+\subset\mathfrak{a}_G$ be a closed Weyl chamber, and let $\lambda_G : G \rightarrow \mathfrak{a}_G^+$ be the associated Jordan projection. Let $\Gamma\subset G$ be a Zariski-dense sub-semi-group. Then the additive group generated by $\lambda_G(\Gamma)$ is dense in $\mathfrak{a}_G$.
\end{fait}

\begin{cor}\label{densitytranslationlengthsbis}
 Let $\Gamma\subset\SL(V)$ be a sub-semi-group whose Zariski closure in $\SL(V)$ is irreducible, semi-simple and non-compact. Then 
 \[\overline{\langle \ell(\gamma), \ \gamma\in\Gamma\rangle}=\R,\]
where $\ell(\gamma)$ is given by \eqref{l(g)}.
\end{cor}

\begin{proof}
 Recall that $V=\R^{d+1}$ so that $\SL(V)$ identifies with $\SL_{d+1}(\R)$. Let $\mathfrak{a}_{\SL(V)}$ (\resp $\mathfrak{a}_{\SL(V)}^+$) be the set of diagonal matrices (\resp diagonal matrices with non-increasing entries) in the Lie algebra of $\SL(V)$, and denote by $\mathfrak{o}:\mathfrak a_{\SL(V)}\rightarrow \mathfrak a_{\SL(V)}^+$ the reordering of the diagonal entries ($\mathfrak o=\lambda_{\SL(V)}\circ\exp$). Denote by $\epsilon_i:\mathfrak{a}_{\SL(V)}\rightarrow\R$ the linear form which gives the $i$-th entry of the diagonal for $i=1,\dots,d+1$. Denote by $G$ the Zariski closure of $\Gamma$ in $\SL(V)$, and $\rho:G\rightarrow \SL(V)$ the inclusion. The point is that we can choose a Cartan subspace $\mathfrak{a}_G$ of $G$ such that $\mathrm{d}\rho(\mathfrak{a}_G)\subset \mathfrak{a}_{\SL(V)}$, but we cannot always choose a Weyl chamber $\mathfrak{a}_G^+\subset\mathfrak{a}_G$ such that $\mathrm{d}\rho(\mathfrak{a}_G^+)\subset\mathfrak{a}_{\SL(V)}^+$. In other words, $\mathfrak{o}\circ\mathrm d\rho : \mathfrak a_G^+\rightarrow\mathfrak a_{\SL(V)}^+$ is not always the restriction of a linear map, and the additive subgroup of $\mathfrak{a}_{\SL(V)}$ generated by $\lambda_{\SL(V)}(\Gamma)$ is not always the image under $\mathrm{d}\rho$ of the additive subgroup of $\mathfrak{a}_G$ generated by $\lambda_G(\Gamma)$.
 
 It happens, however, that, for $i=1$ (\resp $i=d+1$), the map $\epsilon_i\circ\mathfrak{o}\circ\mathrm d\rho:\mathfrak a_G^+ \rightarrow \R$ is the restriction of a linear form, namely the highest weight $\chi_+$ of the representation $\rho$ of $G$ (\resp the highest weight $\chi_-$ of the dual representation in $\SL(V^*)$). As a consequence, the group generated by $\ell(\Gamma)$ is the image under the linear form $\frac12 (\chi_+-\chi_-)\circ\mathrm d\rho$ of the subgroup of $\mathfrak{a}_G$ generated by $\lambda_G(\Gamma)$, which is dense by Fact~\ref{PropAsymp2}.
\end{proof}

Note that Corollary~\ref{densitytranslationlengthsbis} can (and will) be applied, not only to subgroups, but to sub-semi-groups. This is the key observation that allows us to slightly shorten Bray's proof of topological mixing.

\subsection{Schottky subgroups of strongly irreducible groups}\label{Construction of a suitable free sub-semi-group}

In this section, we prove three lemmas that enable us, in the next section, to prove Proposition~\ref{local_non_arithmeticity} by constructing, in any strongly irreducible group of projective transformations containing at least one biproximal element, a strongly irreducible Schottky subgroup made of biproximal elements whose axis are well controlled. The idea is the following: using strong irreducibility and Lemma~\ref{lem:critere irred forte bis}, we find a family $\gamma_1,\dots,\gamma_k$ of biproximal elements that satisfy a finite number of algebraic conditions (see \ref{item:proxspan}, \ref{item:dualproxspan} and \ref{item:posgale} below); then we consider large powers of $\gamma_1,\dots,\gamma_k$ and we show, using the algebraic conditions, that they generate a group which is strongly irreducible (this is elementary, see Lemma~\ref{lem:critere irred forte}), and made of biproximal elements whose axis are well controlled (Lemma~\ref{freesubgroupbis}).

\begin{lemma} \label{lem:critere irred forte}
 Let $k\geq 2$ and let $\gamma_1,\dots,\gamma_k\in \GL(V)$ be biproximal elements such that:
 \begin{enumerate}[label=(\alph*)]
  \item \label{item:proxspan} $\Span(x_{\gamma_i}^\pm, \ 1\leq i\leq k)=V$,
  \item \label{item:dualproxspan} $\bigcap_{1\leq i\leq k}x_{\gamma_i}^0=0$,
  \item \label{item:posgale} $x_{\gamma_i}^\alpha\not\subset x_{\gamma_j}^\beta\oplus x_{\gamma_j}^0$ for any $1\leq i\neq j\leq k$ and $\alpha,\beta\in\{\pm\}$.
 \end{enumerate}
Then the group $\Gamma:=\langle \gamma_i, 1\leq i\leq k\rangle$ generated by $\gamma_1,\dots \gamma_k$ acts strongly irreducibly on $V$.
\end{lemma}

\begin{proof}
 Let us first check that the action of $\Gamma$ on $V$ is irreducible. Consider a non-zero linear subspace $W\subset V$ which is stable under $\Gamma$, and a non-zero vector $w\in W$. Using assumption \ref{item:dualproxspan}, we can find $i$ such that $w\not\in x_{\gamma_i}^0$, and then $\alpha=\pm$ such that $w\not\in x_{\gamma_i}^{-\alpha}\oplus x_{\gamma_i}^0$, so that the sequence $(\gamma_i^{\alpha n}[w])_n$ converges to $x_{\gamma_i}^\alpha$. This means that $x_{\gamma_i}^\alpha\subset W$. Similarly, for any $j\neq i$ and $\beta=\pm$, because $x_{\gamma_i}^\alpha\not\in x_{\gamma_j}^{-\beta}\oplus x_{\gamma_j}^0$ (assumption \ref{item:posgale}), we have $x_{\gamma_j}^\beta\subset W$. Since $k\geq 2$ we deduce that $x_{\gamma_i}^{-\alpha}\subset W$. By assumption \ref{item:proxspan} this means that $W=V$.
 
 Now let $\Gamma_1\subset\Gamma$ be a finite-index subgroup. There exists an integer $N>0$ such that $\Gamma_1$ contains $\Gamma_2:=\langle \gamma_i^N, 1\leq i\leq k\rangle$. The family $\gamma_1^N,\dots,\gamma_k^N$ satisfies conditions \ref{item:proxspan}, \ref{item:dualproxspan} and \ref{item:posgale}, therefore the action of $\Gamma_2$ on $V$ is irreducible, and so is that of $\Gamma_1$: we have proved that $\Gamma$ is strongly irreducible.
\end{proof}

The following can be seen as a converse to Lemma~\ref{lem:critere irred forte}, for groups containing a biproximal element.

\begin{lemma}\label{lem:critere irred forte bis}
 Let $\Gamma\subset\GL(V)$ be a strongly irreducible subgroup that contains a biproximal element $\gamma_1\in\Gamma$. Then there exist biproximal elements $\gamma_2,\dots,\gamma_k\in \Gamma$, with $k\geq 2$, such that the family $\gamma_1,\dots,\gamma_k$ satisfies conditions \ref{item:proxspan},\ref{item:dualproxspan} and \ref{item:posgale} of Lemma~\ref{lem:critere irred forte}.
\end{lemma}

\begin{proof}
 We are going to construct elements $\gamma_i$ inductively, taking conjugates of $\gamma:=\gamma_1$.
 
 Let $k\geq 1$ and $\gamma_2,\dots,\gamma_k$ be such that $\gamma_1,\dots,\gamma_k$ satisfy condition \ref{item:posgale}. If $\Span(x_{\gamma_i}^\pm, 1\leq i\leq k)=V$ and $\bigcap_ix_{\gamma_i}^0=0$ then we are done. Otherwise, we are looking for $g\in\Gamma$ such that:
 \begin{itemize}
  \item $x_{g\gamma g^{-1}}^\alpha=g(x_\gamma^\alpha) \not\in x_{\gamma_i}^\beta\oplus x_{\gamma_i}^0$ for all $i\leq k$ and $\alpha,\beta$ in $\{\pm\}$,
  \item $x_{\gamma_j}^\beta\not\in g(x_\gamma^\alpha\oplus x_\gamma^0)$ for all $i\leq k$ and $\alpha,\beta$ in $\{\pm\}$, 
  \item $gx_\gamma^\pm\not\in\Span(x_{\gamma_i}^\alpha, \ i\leq k,\alpha=\pm)$ if $\Span(x_{\gamma_i}^\alpha, \ i\leq k,\alpha=\pm)\neq V$,
  \item and $\bigcap_{i\leq k}x_{\gamma_i}^0\not\subset g(x_\gamma^0)$ if $\bigcap_{i\leq k}x_{\gamma_i}^0\neq 0$.
 \end{itemize}
Denote by $\Gamma_0$ the Zariski-irreducible component of the identity in $\Gamma$ for the Zariski topology (\ie the topology induced by the Zariski topology on $\GL(V)$), which is a normal finite-index subgroup of $\Gamma$ (see \eg \cite[Lem.\,6.21]{benoist2016random}). The action of $\Gamma_0$ on $V$ is also strongly irreducible. Note that each condition above is Zariski-open with respect to $g\in\Gamma_0$, and non-empty since $\Gamma_0$ acts irreducibly on $V$. That $\Gamma_0$ is Zariski-irreducible implies that the (finite) intersection of all conditions is still non-empty; take an element $g$ inside and set $\gamma_{k+1}:=g\gamma g^{-1}$. Note that conditions \ref{item:proxspan} and \ref{item:dualproxspan} ensure that the process will eventually stop.
\end{proof}

\begin{lemma} \label{freesubgroupbis}
 Let $k\geq 2$ and consider biproximal elements $\gamma_1,\dots,\gamma_k\in \PGL(V)$ which satisfy condition \ref{item:posgale} of Lemma~\ref{lem:critere irred forte}. Consider also a family $\{U_i^\alpha: 1\leq i\leq k \text{ and } \alpha=\pm\}$ of disjoint compact neighbourhoods in $\PR(V)$ of the points $\{x_{\gamma_i}^\alpha: 1\leq i\leq k \text{ and } \alpha=\pm\}$. Then for $N\in\N$ large enough, $\gamma_1^N,\dots,\gamma_k^N$ form the basis of a discrete non-abelian free subgroup of $\PGL(V)$, whose elements are biproximal (apart from the identity), and $x_\gamma^+\in U_{i_1}^{\alpha_1}$ and $x_\gamma^-\in U_{i_n}^{-\alpha_n}$ for any non-trivial cyclically reduced word $\gamma=\gamma_{i_1}^{\alpha_1N}\dots \gamma_{i_n}^{\alpha_nN}$, where  $1\leq i_j\leq k$ and $\alpha_j\in\{\pm\}$ for $1\leq j\leq n$, with $n\geq 1$.
\end{lemma}

To prove this we need a technical fact:

\begin{fait}\label{contract}
 Consider a sequence $(g_n)_n$ in $\PGL(V)$, a point $x$ in $\PR(V)$ and a compact neighbourhood $U$ of $x$ such that the sequence $(g_n(U))_n$ converges to $x$. Then the accumulation points of $(g_n)_n$ in $\PR(\End(V))$ are rank-1 projectors onto $x$ whose kernel does not intersect the interior of $U$.
\end{fait}
 
\begin{proof} 
 Recall that $V=\R^{d+1}$, so that we can use the usual Cartan decomposition in $\GL(V)$: for any $n$ we can write $g_n=[k_na_nl_n]$, where the elements $k_n$ and $\ell_n$ are in the (maximal compact) classical orthogonal subgroup $K=\mathrm{O}(d+1)$ of $\GL_{d+1}(\R)$, and $a_n$ is diagonal with positive non-increasing entries. Up to passing to a subsequence we may assume that $(k_n)_n$ and $(\ell_n)_n$ converge in $K$ with respective limits $k$ and $\ell$. That $(g_n)_n$ contracts an open set to a point implies that $[a_n]_n$ must converge to $[p]$, where $p$ is the projector onto the line spanned by the first vector $e_1$ of the canonical basis. Assume without loss of generality that $(a_n)_n$ converges to $p$, so that $(k_na_n\ell_n)_n$ converges to $kp \ell$, a rank-$1$ matrix with image $k(e_1)$ and kernel $\ell^{-1}\Ker(p)$. The assumption that $(k_na_nl_n(U))_n$ tends to $x$ implies that $(a_nl(U))_n$ goes to $k^{-1}(x)$, which in turn implies that $\ell(U)\cap \Ker(p)=\emptyset$ and $k^{-1}x=e_1$.
\end{proof}

\begin{proof}[Proof of Lemma~\ref{freesubgroupbis}] 
 Let $U$ be a compact subset of $\PR(V)$ with non-empty interior, disjoint from the $U_i^\alpha$'s and the $x_{\gamma_i}^\alpha\oplus x_{\gamma_i}^0$'s. If $N$ is large enough then for any $i=1,\dots,k$ and $\alpha=\pm$,
 \[\gamma_i^{\alpha N}(U\cup U_i^\alpha\cup \bigcup_{\substack{ j\neq i\\ \beta=\pm}}U_j^\beta)\subset U_i^\alpha.\]
 Then $\Gamma'=\langle \gamma_i^N, \ 1\leq i\leq k\rangle$ is free and discrete thanks to a ping-pong argument. Indeed one can prove by induction on word length that for any non-trivial reduced word $\gamma_{i_1}^{\alpha_1N}\dots \gamma_{i_p}^{\alpha_pN}\in\Gamma'$:
 \[\gamma_{i_1}^{\alpha_1N}\dots \gamma_{i_p}^{\alpha_pN}(U\cup U_{i_p}^{\alpha_p}\cup \bigcup_{\substack{j\neq i_p\\ \beta=\pm}}U_j^\beta)\subset U_{i_1}^{\alpha_1}.\]
 We conclude with a proof by contradiction. Assume that there exists a sequence of cyclically reduced words
 \[\gamma_n=\gamma_{i^{(n)}_1}^{\alpha_1^{(n)}N_n}\gamma_{i_2^{(n)}}^{\alpha_2^{(n)}N_n}\dots \gamma_{i_{p_n}^{(n)}}^{\alpha_{p_n}^{(n)}N_n}\]
 with $(N_n)_n$ going to infinity, which are not biproximal with attracting/repelling pair in $U_{i^{(n)}_1}^{\alpha^{(n)}_1}\times U_{i^{(n)}_{p_n}}^{-\alpha^{(n)}_n}$. Up to extracting assume that $i_1^{(n)}=i$, $i_{p_n}^{(n)}=j$, $\alpha_1^{(n)}=\alpha$ and $\alpha_{p_n}^{(n)}=\beta$ do not depend on $n$.
 
 The sequences $(\gamma_n)_n$ and $(\gamma_n^{-1})_n$ respectively contract $U_i^\alpha$ towards $x_{\gamma_i}^\alpha$ and $U_j^{-\beta}$ towards $x_{\gamma_j}^{-\beta}$, so we can apply Fact~\ref{contract} to them. Up to passing to a subsequence we can assume that $(\gamma_n)_n$ and $(\gamma_n^{-1})_n$ converge to rank-1 projectors on respectively $x_{\gamma_i}^\alpha$ and on $x_{\gamma_j}^{-\beta}$. By Remark~\ref{prox_are_open}, the elements $\gamma_n$ and $\gamma_n^{-1}$ are proximal with attracting fixed points respectively in $U_{i}^{\alpha}$ and $U_j^{-\beta}$: this is a contradiction.
\end{proof}

\subsection{Proof of Proposition~\ref{local_non_arithmeticity}}

Up to taking a finite-index subgroup we can assume that $\Gamma\subset\PSL(V)$. Consider its preimage $\widetilde{\Gamma}\subset\SL(V)$. Let $\widetilde{U}\subset T^1\Omega$ be the preimage of $U$. 

By Fact~\ref{densitybis}, there exists a biproximal element $\gamma_1\in\widetilde\Gamma$ whose axis meets $\widetilde{U}$. By Lemma~\ref{lem:critere irred forte bis} we can find biproximal elements $\gamma_2,\dots,\gamma_k\in\widetilde\Gamma$ such that the family $\gamma_1,\dots,\gamma_k$ satisfies conditions \ref{item:proxspan}, \ref{item:dualproxspan} and \ref{item:posgale} of Lemma~\ref{lem:critere irred forte}. 

By Lemma~\ref{freesubgroupbis}, we can find $N\geq 1$ such that $\gamma_1^N,\dots,\gamma_k^N$ generate a non-abelian free subgroup $\Gamma'$ of $\widetilde\Gamma$ whose elements are biproximal, and such that the axis of every element of the form $\gamma_1^N\gamma_{i_1}^N\cdots\gamma_{i_n}^N\gamma_1^N$ meets $\widetilde U$, where $n\geq 0$ and $1\leq i_1,\dots,i_n\leq k$. Moreover, $\Gamma'$ is strongly irreducible by Lemma~\ref{lem:critere irred forte}.

Let $\Gamma^+\subset\Gamma'$ be the sub-semi-group generated by $\gamma_1^N,\gamma_1^N\gamma_2^N\gamma_1^N,\dots,\gamma_1^N\gamma_k^N\gamma_1^N$; it generates $\Gamma'$ as a group. By Fact~\ref{zclosuresemisimple}, the Zariski-closure of $\Gamma^+$ (which is also that of $\Gamma'$) is semi-simple and non-compact. By Corollary~\ref{densitytranslationlengthsbis}, the additive group generated by $\{\ell(\gamma): \gamma\in\Gamma^+\}$ is dense in $\R$. By construction, each non-trivial element $\gamma$ of $\Gamma^+$ is biproximal with axis \emph{through $\widetilde{U}$}; by definition this axis projects on a biproximal periodic geodesic \emph{through $U$}, whose length is $\ell(\gamma)$.

\section{Strong stable manifolds}\label{Strong stable manifolds}

The \emph{strong stable manifold} of a vector $v\in T^1\Omega$ is a classical notion in the theory of dynamical systems; it is defined as the set of vectors $w\in T^1\Omega$ such that $d_{T^1\Omega}(\phi_t v,\phi_tw)$ goes to zero as $t$ goes to infinity. The goal of this section is to establish the following geometric description of the strong stable manifolds centred at smooth points.

\begin{prop}\label{strongstablemanifolds}
Let $\Omega\subset\PR(V)$ be a properly convex open set, let $v\in T^1\Omega$ and let $\xi:=\phi_\infty v\in\partial\Omega$. Then
\begin{enumerate}
 \item \label{Item : varstabfaible} for any $w\in T^1\Omega$ such that $\phi_\infty w=\xi$, the function $t\mapsto d_{T^1\Omega}(\phi_tv,\phi_tw)$ is non-increasing;
 \item \label{Item : variete stable} suppose $\xi$ is smooth. Then for any $w\in T^1\Omega$ with $\phi_\infty w=\xi$, there is a unique $t_0\in\R$ for which the lines $\phi_{-\infty} v\oplus\phi_{-\infty}w$ and $\pi v\oplus\pi\phi_{t_0} w$ intersect at a point of $T_\xi\partial\Omega$ (see Figure~\ref{figure_strongstablemanifold}); moreover $t_0$ is the only number for which $v$ and $\phi_{t_0}w$ are on the same strong stable manifold, \ie
 \[d_{T^1\Omega}(\phi_tv,\phi_{t+t_0}w)\underset{t\to\infty}{\longrightarrow}0.\]
\end{enumerate}
\end{prop}

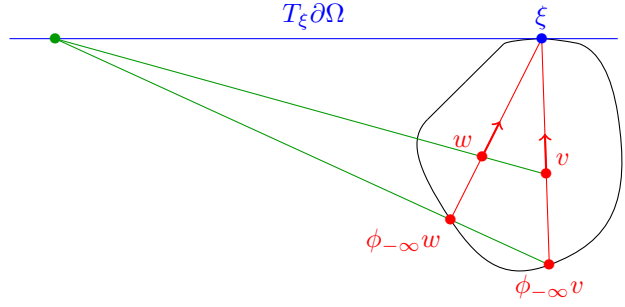
\begin{figure}
\centering
\begin{tikzpicture}
 \coordinate (xi) at (0,0);
 \coordinate (h) at (-7,0);
 \coordinate (h') at (1,0);
 \coordinate (phiv) at (0.1,-3);
 \coordinate (phiw) at (-1.2,-2.4);
 \coordinate (v) at ($(xi)!0.6!(phiv)$);
 \coordinate (v') at ($(v)!0.3!(xi)$);
 \coordinate (x) at (intersection of xi--h and phiv--phiw);
 \coordinate (w) at (intersection of xi--phiw and v--x);
 \coordinate (w') at ($(w)!0.3!(xi)$);
 \coordinate (A) at (-0.5,-0.07);
 \coordinate (B) at (-1,-0.57);
 \coordinate (C) at (-1.5,-1.07);
 \coordinate (D) at (1,-1);
 \coordinate (E) at (0.5,-0.07);
 
 \length{A,B,C,phiw,phiv,D,E,xi}
 \cvx{A,B,C,phiw,phiv,D,E,xi}{1.5}
 
 \draw [blue] (h')--(h) node[midway,above,blue] {$T_\xi\partial\Omega$};
 \draw [red] (phiv)--(xi);
 \draw [red] (phiw)--(xi);
 \draw [red,thick,->] (v)--(v');
 \draw [red,thick,->] (w)--(w');
 \draw [green!60!black] (phiv)--(x);
 \draw [green!60!black] (v)--(x);
 \draw (x) node[green!60!black]{$\bullet$};
 \draw (xi) node[blue]{$\bullet$} node[above,blue]{$\xi$};
 \draw (v) node[red]{$\bullet$} node[above right,red]{$v$};
 \draw (w) node[red]{$\bullet$} node[above left,red]{$w$};
 \draw (phiv) node[red]{$\bullet$} node[below,red]{$\phi_{-\infty}v$};
 \draw (phiw) node[red]{$\bullet$} node[below left,red]{$\phi_{-\infty}w$};
\end{tikzpicture}
\caption{Vectors $v,w\in T^1\Omega$ in the same strong stable manifold}\label{figure_strongstablemanifold}
\end{figure}

\subsection{Crampon's Lemma}

For a general properly convex open set $\Omega\subset\PR(V)$, it is not true that the distance function is convex, in the sense that $t\mapsto d_\Omega(c_1(t),c_2(t))$ is convex for all geodesics $c_1,c_2$ (see \cite{socie-methou04}). However, one can establish a weaker property, and this is the subject of the next lemma. Observe that it implies in particular the first part of Proposition~\ref{strongstablemanifolds}. We will give a proof of Lemma~\ref{crampon} in the appendix, to clarify a missing detail in Crampon's original proof.

\begin{lemma}[\!{\!\cite[Lem.\,8.3]{crampon}}]\label{crampon}
Let $\Omega$ be a properly convex open subset of $\PR(V)$. Let $c_1$ and $c_2$ be two straight geodesics parametrised with constant speed, but not necessarily with the same speed. Then for all $0\leq t\leq T$,
\[d_{\Omega}(c_1(t),c_2(t))\leq d_\Omega(c_1(0),c_2(0))+d_\Omega(c_1(T),c_2(T)).\]
\end{lemma}

\subsection{An explicit computation of \texorpdfstring{$\lim_{t\to\infty}d_{T^1\Omega}(\phi_tv,\phi_tw)$}{lim as t to infty of d(phitv,phitw)}}

We now prove a lemma from which the geometric description of strong stable manifolds will be a corollary. See an illustration for the notation in Figure~\ref{figure_ctrlvarstab}.

\begin{lemma}\label{ctrlvarstab}
Let $\Omega\subset \PR(V)$ be a properly convex open set. Take $v,w\in T^1\Omega$ with $\phi_\infty v=\phi_\infty w=:\xi$ and $\phi_{-\infty}v\neq\phi_{-\infty}w$, and let $x:=\pi v$ and $y:=\pi w$. Let $P$ be the projective plane spanned by $x$, $y$ and $\xi$, and let $\mathcal{D},\mathcal{D}'$ be the tangent lines of $P\cap\partial\Omega$ at $\xi$ (they may coincide) such that the lines $\mathcal{D},(\xi \oplus x),(\xi \oplus y),\mathcal{D}'$ lie in this order around $\xi$. Let $\delta\geq0$ be half the logarithm of the cross-ratio of these four lines (with the convention that $\delta=0$ if $\mathcal{D}=\mathcal{D}'$). Let $a$ be the intersection point of $x\oplus y$ and $\phi_{-\infty}v\oplus \phi_{-\infty}w$. If $a\oplus\xi$ does not intersect $\Omega$, then
\[d_\Omega(\pi\phi_tv,\pi\phi_tw) \underset{t\to\infty}{\longrightarrow}\delta \ \text{ and }\ d_{T^1\Omega}(\phi_tv,\phi_tw) \underset{t\to\infty}{\longrightarrow}\delta.\]
\end{lemma}
Observe that, in Lemma~\ref{ctrlvarstab}, we have $\mathcal{D}=\mathcal{D'}$ whenever $P\cap\partial\Omega$ is smooth at $\xi$, in which case $\delta=0$.

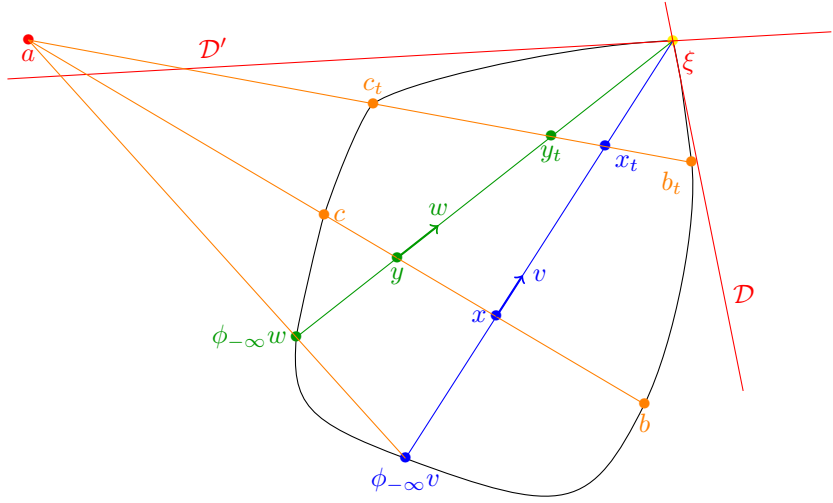
\begin{figure}
\centering
\begin{tikzpicture}[scale=2,rotate=20]
 \coordinate (xi) at (0,0);
 \coordinate (eta1) at (-2.6,-2);
 \coordinate (eta2) at (-3,-1);
 \coordinate (b) at (-1,-2.2);
 \coordinate (bt) at (-0.16,-0.8);
 \coordinate (c) at (-2.55,-0.3);
 \coordinate (a) at (intersection of eta1--eta2 and b--c);
 \coordinate (ct) at ($(a)!0.52!(bt)$);
 \coordinate (x) at (intersection of b--c and xi--eta1);
 \coordinate (xt) at (intersection of bt--ct and xi--eta1);
 \coordinate (y) at (intersection of b--c and xi--eta2);
 \coordinate (yt) at (intersection of bt--ct and xi--eta2);
 \coordinate (A) at (-0.02,-0.14);
 \coordinate (B) at (-0.04,-0.26);
 \coordinate (C) at (-0.2,0.06);
 \coordinate (D) at (-0.1,0.032);
 \coordinate (v) at ($(x)!0.15!(xi)$);
 \coordinate (w) at ($(y)!0.15!(xi)$);
 
 \length{xi,D,C,ct,c,eta2,eta1,b,bt,B,A}
 \cvx{xi,D,C,ct,c,eta2,eta1,b,bt,B,A}{1.8}
 
 \draw (xi) node[yellow]{$\bullet$} node[below right,red]{$\xi$};
 \draw (eta1) node[blue]{$\bullet$} node[below,blue]{$\phi_{-\infty}v$};
 \draw (eta2) node[green!60!black]{$\bullet$} node[left,green!60!black]{$\phi_{-\infty}w$};
 \draw (b) node[orange]{$\bullet$} node[below,orange]{$b$};
 \draw (bt) node[orange]{$\bullet$} node[below left,orange]{$b_t$};
 \draw (c) node[orange]{$\bullet$} node[right,orange]{$c$};
 \draw (ct) node[orange]{$\bullet$} node[above,orange]{$c_t$};
 \draw (a) node[red]{$\bullet$} node[below,red]{$a$};
 \draw [green!60!black] (eta2)--(xi);
 \draw [blue] (eta1)--(xi);
 \draw (y) node[green!60!black]{$\bullet$} node[below,green!60!black]{$y$};
 \draw (x) node[blue]{$\bullet$} node[blue,left]{$x$};
 \draw (yt) node[green!60!black]{$\bullet$} node[green!60!black,below]{$y_t$};
 \draw (xt) node[blue]{$\bullet$} node[blue,below right]{$x_t$};
 \draww{xi}{B}{1}{8}{red}{red,near end,right}{$\mathcal{D}$}
 \draww{xi}{C}{5}{20}{red}{red,near end,above}{$\mathcal{D}'$}
 \draw [orange] (a)--(b);
 \draw [orange] (a)--(bt);
 \draw [orange] (a)--(eta1);
 \draw [blue, thick, ->] (x)--(v) node[right]{$v$};
 \draw [green!60!black, thick, ->] (y)--(w) node[above]{$w$};
\end{tikzpicture}
\caption{Illustration for the notation in Lemma~\ref{ctrlvarstab}}\label{figure_ctrlvarstab}
\end{figure}

\begin{proof}
We consider $x_t=\pi\phi_tv$ and $y_t=\pi\phi_tw$. Since $d_\Omega(x,x_t)=t=d_\Omega(y,y_t)$ and by definition of the cross-ratio, we see that $y_t\in (y\oplus \xi)\cap(a\oplus x_t)$. Let $b_t,c_t\in\partial\Omega$ be such that the four points $b_t,x_t,y_t,c_t$ are aligned in this order. We consider $\mathcal{D}_t=(\xi \oplus b_t)$ and $\mathcal{D}'_t=(\xi \oplus c_t)$. By definition of the tangent lines, the two sequences $(\mathcal{D}_t)_{t\to\infty}$ and $(\mathcal{D}'_t)_{t\to\infty}$ converge respectively to $\mathcal{D}$ and $\mathcal{D}'$. By definition $d_\Omega(x_t,y_t)$ is half the logarithm of the cross-ratio of the four lines $\mathcal{D}_t,(\xi\oplus x),(\xi\oplus y),\mathcal{D}'_t$, which converges to the cross-ratio of the four lines $\mathcal{D},(\xi\oplus x),(\xi\oplus y),\mathcal{D}'$.
\end{proof}

\subsection{Proof of Proposition~\ref{strongstablemanifolds}}
\begin{itemize}
\item[\ref{Item : varstabfaible}.] By definition of $d_{T^1\Omega}$ (see \eqref{eq:dT1Om}), in order to prove that $t\mapsto d_{T^1\Omega}(\phi_tv,\phi_tw)$ is non-increasing, it is enough to prove that $t\mapsto d_\Omega(\pi\phi_tv,\pi\phi_tw)$ is non-increasing. Observe that it will also have as a consequence that $d_{T^1\Omega}(v,w)\leq d_\Omega(\pi v,\pi w)$. We fix $t\geq 0$. Consider a sequence $(x_n)_{n\in\N}\in\Omega^\N$ converging to $\xi$, and for each $n\in\N$, take $v_n\in T_{\pi v}^1\Omega$ and $w_n\in T_{\pi w}^1\Omega$ which define geodesic rays containing $x_n$. Then $\phi_t(v)=\lim_{n\to\infty}\phi_t(v_n)$ and $\phi_t(w)=\lim_{n\to\infty}\phi_t(w_n)$. By Lemma~\ref{crampon},
\[d_\Omega(\pi\phi_tv_n,\pi\phi_{t\frac{d_\Omega(\pi w,x_n)}{d_\Omega(\pi v,x_n)}} w_n)\leq d_\Omega(\pi v,\pi w)\]
for any $n$, and we get the desired inequality by taking the limit, since $(\frac{d_\Omega(\pi w,x_n)}{d_\Omega(\pi v,x_n)})_n$ tends to $1$.

\item[\ref{Item : variete stable}.] If $\xi$ is a smooth point of $\partial\Omega$ and the lines $\phi_{-\infty}v\oplus\phi_{-\infty}w$ and $\pi v\oplus\pi w$ intersect at a point of $T_\xi\partial\Omega$, then the fact that $d_{T^1\Omega}(\phi_tv,\phi_tw)$ goes to zero as $t$ goes to infinity, is an immediate corollary of Lemma~\ref{ctrlvarstab} (note that in this case $\mathcal{D}=\mathcal{D'}$ with the notation of Lemma~\ref{ctrlvarstab}).
\end{itemize}

\subsection{A consequence}\label{Section : une consequence}

The following lemma is a consequence of Proposition~\ref{strongstablemanifolds}.\ref{Item : varstabfaible}. It will be the key ingredient to prove that the biproximal unit tangent bundle is maximal among invariant closed subsets on which the geodesic flow is topologically transitive (Theorem~\ref{mixing}.\ref{item:max de T1Mbip}).

\begin{lemma}\label{Lemme : limite d'orbite}
 Let $\Omega\subset\PR(V)$ be a properly convex open set and $\Gamma\subset\Aut(\Omega)$ a discrete subgroup; set $M=\Omega/\Gamma$. Consider two vectors $v,w\in T^1M$ with $w$ in the closure of the forward orbit $\{\phi_tv:t\geq 0\}$. Then $\phi_\infty\tilde{v}$ belongs to the closure of the orbit $\Gamma\cdot \phi_\infty\tilde{w}$ for all lifts $\tilde{v},\tilde{w}\in T^1\Omega$. 
\end{lemma}

\begin{proof}
 By assumption there exist sequences $(t_n)_n\in [0,\infty)^\N$ and $(\gamma_n)_n\in\Gamma^\N$ such that 
 \[\gamma_n\phi_{t_n}\tilde{v}\underset{n\to\infty}{\longrightarrow}\tilde{w}.\]
 This implies that, for $n$ large enough, $[\gamma_n\phi_{-\infty}\tilde{v},\phi_\infty\tilde{w}]\cap\Omega$ is non-empty; let us consider $\tilde{u}_n\in T^1\Omega$ such that $\phi_{-\infty}\tilde{u}_n=\gamma_n\phi_{-\infty}\tilde{v}$, and $\phi_\infty\tilde{u}_n=\phi_\infty\tilde{w}$, and $\pi\tilde{u}_n$ is a closest point of $[\gamma_n\phi_{-\infty}\tilde{v},\phi_\infty\tilde{w}]$ to $\pi\tilde{w}$ for the Hilbert distance. We easily observe that $(\tilde{u}_n)_n$ converges to $\tilde{w}$ as $n$ tends to infinity. By Proposition~\ref{strongstablemanifolds}.\ref{Item : varstabfaible}, we obtain
 \begin{align*}
  d_{T^1\Omega}(\tilde{v},\gamma_n^{-1}\phi_{-t_n}\tilde{u}_n) & \leq d_{T^1\Omega}(\phi_{t_n}\tilde{v},\gamma_n^{-1}\tilde{u}_n) = d_{T^1\Omega}(\gamma_n\phi_{t_n}\tilde{v},\tilde{u}_n) \\
  & \leq d_{T^1\Omega}(\gamma_n\phi_{t_n}\tilde{v},\tilde{w}) + d_{T^1\Omega}(\tilde{w},\tilde{u}_n) \\
  & \underset{n\to\infty}{\longrightarrow} 0.
 \end{align*}
 Therefore, $\gamma_n^{-1}\phi_\infty\tilde{w}=\gamma_n^{-1}\phi_\infty \tilde{u}_n$ tends to $\phi_\infty \tilde{v}$ as $n$ goes to infinity.
\end{proof}

\section{Proof of Theorem~\ref{mixing}}\label{Proof of Mixing}

\subsection{Topological mixing on the biproximal unit tangent bundle}

The first claim of Theorem~\ref{mixing} is an immediate consequence of Proposition~\ref{local_non_arithmeticity} and the following.

\begin{prop}
 Let $\Omega\subset\PR(V)$ be a properly convex open set and $\Gamma\subset\Aut(\Omega)$ a discrete subgroup. Set $M=\Omega/\Gamma$ and assume that $T^1M_{bip}$ is non-empty and that the set of lengths of the biproximal periodic orbits through any non-empty open subset $U\subset T^1M_{bip}$ generate a dense subgroup of $\R$. Suppose moreover that $\Lambda_\Gamma\smallsetminus((x_\gamma^+\oplus x_\gamma^0)\cup (x_\gamma^-\oplus x_\gamma^0))$ is non-empty for any biproximal element $\gamma\in\Gamma$. Then the geodesic flow on $T^1M_{bip}$ is topologically mixing.
\end{prop}

\begin{proof}
Let $U_1$ and $U_2$ be two open subsets of $T^1M_{bip}$. Let us prove that there exists $T>0$ such that $\phi_t(U_1)\cap U_2\neq\emptyset$ for all $t\geq T$.

Since the map $(t,v)\mapsto \phi_t(v)$ is continuous, we can find an open subset $\emptyset\neq U_2'\subset U_2$, and $\epsilon>0$ such that $\phi_{[-\epsilon,\epsilon]}(U_2')\subset U_2$. As a consequence, for any $t\in\R$, if $\phi_t(U_1)$ and $U_2'$ intersect, then for  all $s\in\R$ which are $\epsilon$-close to $t$, the sets $\phi_s(U_1)$ and $U_2$ intersect.

Let $\pi_\Gamma : T^1\Omega\rightarrow T^1M$ be the natural projection. Let us find small non-empty open subsets $\widetilde{V}_1\subset\pi_\Gamma^{-1}U_1$ and $\widetilde{V}_2\subset \pi_\Gamma^{-1}U_2'$ such that for any $(\tilde{v}_1,\tilde{v}_2)\in \widetilde{V}_1\times\widetilde{V}_2$, the line $(\phi_{-\infty}\tilde{v}_1\oplus\phi_{\infty}\tilde{v}_2)$ intersects $\Omega$.  By assumption, we cand find $\tilde{u}_1\in\pi_\Gamma^{-1}(U_1)$ and $\tilde{u}_2\in\pi_\Gamma^{-1}(U_2')$ biproximal periodic. If $(\phi_{-\infty}\tilde{u}_1\oplus\phi_{\infty}\tilde{u}_2)$ intersects $\Omega$, then we can take $\widetilde{V}_1$ (\resp $\widetilde V_2$) to be a small neighbourhood of $\tilde u_1$ (\resp $\tilde u_2$). Otherwise, $T_{\phi_{-\infty}\tilde u_1}\partial\Omega=T_{\phi_\infty \tilde u_2}\partial\Omega$ since $\phi_{-\infty}\tilde u_1$ and $\phi_\infty \tilde u_2$ are smooth by Lemma~\ref{biproxiattractingpointsaresmooth}. By assumption, there exists $x\in\Lambda_\Gamma\smallsetminus (T_{\phi_{-\infty}\tilde u_2}\partial\Omega\cup T_{\phi_\infty\tilde u_2}\partial\Omega)$. We can then take $\widetilde V_1$ to be a small neighbourhood of $\tilde u_1$, and $\widetilde V_2$ to be a small neighbourhood of any vector $\tilde u_2'$ which is close enough to $\tilde u_2$ and such that $\phi_{-\infty}\tilde u_2'=\phi_{-\infty}\tilde u_2$ and $\phi_\infty\tilde u_2'=\gamma_2^n x$ for $n$ large enough, where $\gamma_2\in\Gamma$ is the biproximal element associated to the orbit of $\tilde u_2$.
 
For $i=1,2$, pick $\tilde{v}_i\in \widetilde{V}_i$ biproximal with period $\tau_i$ such that $\tau_1\Z+\tau_2\Z$ is $\frac{\epsilon}{2}$-dense in $\R$; this is possible by assumption and by Observation~\ref{subgroups_of_R} below.

We know that $(\phi_{-\infty}\tilde{v}_1,\phi_\infty \tilde{v}_2)\subset\Omega$. Let us consider $\tilde{w}\in T^1\Omega$ tangent to this line: it is pointing forward at $\phi_{\infty}\tilde{v}_2$ and backward at $\phi_{-\infty}\tilde{v}_1$. By Lemma~\ref{biproxiattractingpointsaresmooth}, the points $\phi_{\pm\infty}\tilde{v}_i$ are smooth. Using Proposition~\ref{strongstablemanifolds}, we can find $t_1$ and $t_2\in\R$ which verify $\lim_{t\to\infty}d_{T^1\Omega/\Gamma}(\phi_{-t}v_1,\phi_{t_1-t}w)=0$ and $\lim_{t\to\infty}d_{T^1\Omega/\Gamma}(\phi_ tv_2,\phi_{t_2+t}w)=0$, where for $i=1,2$, we set $v_i=\pi_\Gamma\tilde{v}_i$, and $w=\pi_\Gamma\tilde{w}$. This implies that there is an integer $N\geq 0$ such that for any $n\geq N$, $\phi_{t_1-n\tau_1}w\in V_1:=\pi_\Gamma\widetilde{V}_1$ and $\phi_{t_2+n\tau_2}w\in V_2:=\pi_\Gamma\widetilde{V}_2$. We deduce that $\phi_t(V_1)\cap V_2\neq\emptyset$ for any $t$ in $\{-t_1+t_2+n_1\tau_1+n_2\tau_2: \ n_1,n_2\geq N\}$.

We conclude by observing that the $\epsilon$-neighbourhood of $\{-t_1+t_2+n_1\tau_1+n_2\tau_2: \ n_1,n_2\geq N\}$ contains a real interval of the form $[T,\infty)$ for $T$ large enough. Indeed, there exists $N'>0$ such that $[0,\tau_1]$ is contained in the $\epsilon$-neighbourhood of $\{n_1\tau_1+n_2\tau_2: |n_i|\leq N'\}$. Then the $\epsilon$-neighbourhood of $\{-t_1+t_2+n_1\tau_1+n_2\tau_2: n_1,n_2\geq N\}$ contains $[-t_1+t_2+(N+N')\tau_1+(N+N')\tau_2,\infty[$.
\end{proof}

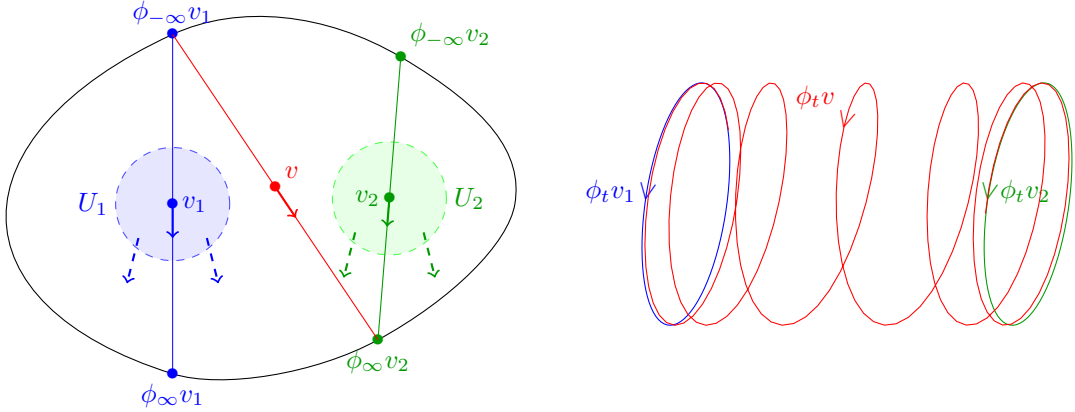
\begin{figure}
\centering
\begin{tikzpicture}[scale=1.5]
 \coordinate (U1) at (0,0);
 \coordinate (U1') at (0,-3);
 \coordinate (U2) at (2,-0.2);
 \coordinate (U2') at (1.8,-2.7);
 \coordinate (v1) at ($(U1)!0.5!(U1')$);
 \coordinate (v1') at ($(v1)!0.2!(U1')$);
 \coordinate (v2) at ($(U2)!0.5!(U2')$);
 \coordinate (v2') at ($(v2)!0.2!(U2')$);
 \coordinate (A) at (3,-1.5);
 \coordinate (v) at ($(U1)!0.5!(U2')$);
 \coordinate (v') at ($(v)!0.2!(U2')$);
 \path (v1) ++(-0.7,0) node[blue]{$U_1$};
 \path (v2) ++(0.7,0) node[green!60!black]{$U_2$};
 
 \length{U1,U1',U2',A,U2}
 \cvx{U1,U1',U2',A,U2}{1.5}
 
 \draw [color=blue!70,dashed,fill=blue!10] (v1) circle (0.5);
 \draw [color=green!70,dashed,fill=green!10] (v2) circle (0.5);
 \draw (U1) node[blue]{$\bullet$} node[above,blue]{$\phi_{-\infty}v_1$} ;
 \draw (U1') node[blue]{$\bullet$} node[below,blue]{$\phi_\infty v_1$} ;
 \draw (U2) node[green!60!black]{$\bullet$} node[above right,green!60!black]{$\phi_{-\infty}v_2$} ;
 \draw (U2') node[green!60!black]{$\bullet$} node[below,green!60!black]{$\phi_\infty v_2$} ;
 \draw (v1) node[blue]{$\bullet$} node[blue,right]{$v_1$};
 \draw [blue, thick,->] (v1)--(v1');
 \draw (v2) node[green!60!black]{$\bullet$} node[green!60!black,left]{$v_2$};
 \draw [green!60!black, thick,->] (v2)--(v2');
 \draw [blue,very thin] (U1)--(U1');
 \draw [green!60!black, thin] (U2)--(U2');
 \draw (v) node[red]{$\bullet$} node[red,above right]{$v$};
 \draw [red,thin] (U1)--(U2');
 \draw [red, thick,->] (v)--(v');
 \draw (v1)++(-0.3,-0.3)[->,dashed,blue,thick]--++(-0.1,-0.4);
 \draw (v1)++(0.3,-0.3)[->,dashed,blue,thick]--++(0.1,-0.4);
 \draw (v2)++(-0.3,-0.3)[->,dashed,green!60!black,thick]--++(-0.1,-0.4);
 \draw (v2)++(0.3,-0.3)[->,dashed,green!60!black,thick]--++(0.1,-0.4);

 \draw [domain=0:2*pi, samples=80,blue,postaction={decorate},
    decoration={
      markings,
      mark=at position .2 with
      {\draw[ultra thick,blue] node[transform shape]{$>$} node[left]{$\phi_tv_1$};}}] plot (4.5, {cos(\x r)-1.5}, {sin(\x r)});
 \draw [domain=0:2*pi, samples=80,green!60!black,postaction={decorate},
    decoration={
      markings,
      mark=at position .2 with
      {\draw[ultra thick,green!60!black] node[transform shape]{$>$} node[right]{$\phi_tv_2$};}}] plot (7.5, {cos(\x r)-1.5}, {sin(\x r)});
 \draw [domain=5:16*pi-5, samples=200,red,postaction={decorate},
    decoration={
      markings,
      mark=at position .5 with
      {\draw[ultra thick,red] node[transform shape]{$>$} node[above left]{$\phi_tv$};}}] plot ({6+1.5*sin(pi/2*sin((-pi/2+\x/16) r) r)}, {cos(\x r)-1.5}, {sin(\x r)});
\end{tikzpicture}
\caption{Proof of mixing. On the left: In $\Omega$. On the right: In the quotient $M=\Omega/\Gamma$.}\label{figure_mixing}
\end{figure}

\begin{obs}\label{subgroups_of_R}
Let $A$ be a subset of $\R$ which generates a dense additive subgroup $G$ of $\R$. Let $x,\epsilon>0$. Then there exists $g\in A$ such that $x\Z+g\Z$ is $\epsilon$-dense in $\R$ (\ie any point in $\R$ is at distance at most $\epsilon$ from $x\Z+g\Z$).
\end{obs}

\begin{proof}
Up to replacing $A$ by $A/x$ and $\epsilon$ by $\epsilon/x$, we can assume that $x=1$. Then there are two possibilities.
\begin{itemize}
\item
The set $A$ contains an irrational number $g$. Then $\Z+g\Z$ is dense in $\R$.
\item
The set $A$ is contained in $\Q$. Let $q_0\in\N^*$ be such that $\frac{1}{q_0}<\epsilon$. The subgroup $\frac{1}{q_0!}\Z$ is not dense in $\R$, so $A$ must contain an element outside of it, of the form $\frac{p}{q}$ with $p$ and $q$ coprime and $q>q_0$. The group $\Z+\frac{p}{q}\Z=\frac{1}{q}\Z$ is $\epsilon$-dense in $\R$.\qedhere
\end{itemize}
\end{proof}

\subsection{Maximality of the biproximal unit tangent bundle}

The second claim of Theorem~\ref{mixing} is a consequence of the following more general proposition.

\begin{prop}
 Let $\Omega\subset\PR(V)$ be a properly convex open set and $\Gamma\subset\Aut(\Omega)$ a discrete subgroup; set $M=\Omega/\Gamma$. If $T^1M_{bip}$ is non-empty, then it is maximal for inclusion among all closed invariant subsets of $T^1M$ on which the geodesic flow is topologically transitive.
\end{prop}

\begin{proof}
Consider a closed invariant subset $A\subset T^1M$ containing $T^1M_{bip}$ and on which $(\phi_t)_t$ is topologically transitive. Since $A$ has the Baire property and is second countable, there exists $v\in A$ such that both $\{\phi_{t}v : t\geq 0\}$ and $\{\phi_tv : t\leq 0\}$ are dense in $A$. 

Take $w\in T^1M_{bip}$ (which is non-empty by assumption), and consider respective lifts $\tilde{v},\tilde{w}\in T^1\Omega$ of $v,w$. By definition of $v$, the vector $w$ is in the closure of the forward orbit $\{\phi_tv : t\geq 0\}$, so we can apply Lemma~\ref{Lemme : limite d'orbite} and we obtain $\phi_\infty\tilde{v}\in\Lambda_\Gamma$. Using again Lemma~\ref{Lemme : limite d'orbite}, and the fact that $w$ is in the closure of $\{\phi_t(-v) : t\geq 0\}=\{-\phi_tv : t\leq 0\}$, we see that $\phi_{-\infty}\tilde{v}\in\Lambda_\Gamma$. We have proved that $v\in T^1M_{bip}$, therefore 
\[A=\overline{\{\phi_tv : t\in\R\}}\subset T^1M_{bip}\subset A,\]
and this concludes the proof.
\end{proof}

\section{The geodesic flow in the higher-rank compact case}\label{symmetric_convex}

The goal of this section is to prove Proposition~\ref{symmetric_top_mixing}. We are actually going to prove a finer statement: that the connected components of the non-wandering set of the geodesic flow are quotients of homogeneous spaces whose Haar measure is mixing.

We denote by $\mathbb{H}$ the classical division algebra of quaternions, and by $\mathbb{O}$ the classical non-associative division algebra of octonions. Fix an integer $N\geq 3$ and the algebra $\K=\R$, $\C$, $\mathbb{H}$, or, if $N=3$, $\mathbb{O}$. 
We shall use the following notation. In the case $\K=\R$, conjugation is the identity and we abusively say Hermitian instead of symmetric.
\begin{itemize}
\item For $x\in\K$, the element $\overline{x}\in\K$ is the conjugate of $x$.
\item We consider the Hermitian bilinear form on $\K^N$ given by $\langle x,y\rangle=\sum_{i=1}^N\overline{x}_iy_i$.
\item The real vector space $V=V_{N,\K}$ consists of the Hermitian matrices of size $N$ with entries in $\K$.
\item The cone $C=C_{N,\K}\subset V$ consists of the positive-definite Hermitian matrices.
\item The properly convex open set $\Omega=\Omega_{N,\K}\subset \PR(V)$ is the projectivisation of $C$.
\item The group $\Aut(C)\subset \GL(V)$ consists of the transformations preserving $C$.
\item The group $G=G_{N,\K}:=\Aut(\Omega)=\Aut(C)/\R^*$ is the automorphism group of $\Omega$, where $\R^*$ is seen as the group of homotheties of $\GL(V)$.
\item The group $K\subset\Aut(C)$ is the stabiliser of the identity matrix; note that the map $K\rightarrow G$ is an embedding, and that $K$ is a maximal compact subgroup of $G$.
\item Finally the group $A$ consists of the diagonal matrices of size $N$ with entries in $\R_{>0}$; we see it embedded in $\Aut(C)$, acting on $V$ by the following formula: $a\cdot X=aXa$ for $a\in A$ and $X\in V$.
\end{itemize}

Let us be more explicit about the case $\K=\R$. The group $\Aut(C_{N,\R})$ identifies with $\GL_N(\R)/\{\pm1\}$, acting on $V_{N,\R}$ by the formula $g\cdot X=gXg^t$; the group $G_{N,\R}$ identifies with $\PGL_N(\R)$; the group $K$ identifies with $\mathrm{O}(N)/\{\pm1\}$.

We come back to the general case. The spectral theorem (see \cite[Th.\,V.2.5]{analysis_on_symmetric_cones}) ensures that for every $X\in V$ there exists $k\in K$ such that $k\cdot X$ is diagonal with real entries. This, using the action of $A$, has two consequences: $\Aut(C)$ acts transitively on $C$, and can be written as the product $KAK=\{k_1ak_2: k_1,k_2\in K, a\in A\}$. Then, the quotient group $G$ acts transitively on $\Omega$, and can be written $K(A/\R_{>0})K$ --- actually, the element of $A/\R_{>0}$ in the decomposition can be taken with non-increasing entries on the diagonal, and this yields the \emph{Cartan decomposition} of $G$. The Lie algebra of $G$ is $\mathfrak{sl}(N,\K)$ when $\K\neq\mathbb{O}$, and $\mathfrak{e}_{6(-26)}$ if $\K=\mathbb{O}$ (see \cite[p.\,97]{analysis_on_symmetric_cones}), therefore $G$ is a non-compact real simple Lie group, with finitely many connected components, and with trivial center. Observe that $\Omega$ identifies as a $G$-space with the Riemannian symmetric space of $G$.

Since $\Omega=G/K$, a discrete subgroup $\Gamma\subset G$ acts cocompactly on $\Omega$ if and only if $G/\Gamma$ is compact, \ie  $\Gamma$ is a uniform lattice of $G$; uniform lattices exist by a theorem of Borel \cite[Th.\,C]{borel_reseau}. The properly convex open sets $\Omega_{N,\K}$ are called the \emph{symmetric divisible convex sets}. Zimmer \cite[Th.\,1.4]{zimmer_higher_rank} recently proved that the higher-rank irreducible closed convex projective manifolds are exactly the quotients of the form $\Omega_{N,\K}/\Gamma$, where $N\geq 3$, the field $\K=\R$, $\C$, $\mathbb{H}$, or $\mathbb{O}$ (if $N=3$), and $\Gamma$ is a uniform lattice of $G_{N,\K}$.

\subsection{The non-wandering set of \texorpdfstring{$G\times\R$}{GNR} on \texorpdfstring{$T^1\Omega$}{T1Omega)}}

In this section we describe $\NW(T^1\Omega,G\times\R)$ and prove that $G$ acts transitively on each of its connected components. 

The boundary of $\Omega$ is the projectivisation of the cone of positive semi-definite Hermitian matrices. For $1\leq i,j\leq N-1$, we denote by $T^1\Omega_{i,j}$ the set of unit tangent vectors $v\in T^1\Omega$ such that the respective ranks of $\phi_{-\infty}v$ and $\phi_\infty v$ (meaning the rank of any representative in $V$) are $i$ and $j$. Note that $T^1\Omega_{i,j}$ is non-empty if and only if $i+j\geq N$ (see Proposition~\ref{NW dans le cas symétrique}.\ref{item:T1Oij vide}). The subsets $T^1\Omega_{i,j}$, for $1\leq i,j\leq N$ and $i+j\geq N$, are invariant under the automorphism group $\Aut(\Omega)$ and the geodesic flow $(\phi_t)_{t\in\R}$. They stratify $T^1\Omega$ in the following way:
\begin{itemize}
\item $T^1\Omega$ is the disjoint union of the $T^1\Omega_{i,j}$,
\item the closure of $T^1\Omega_{i,j}$ is the union of the $T^1\Omega_{k,\ell}$ for $1\leq k\leq i$ and $1\leq \ell\leq j$,
\item in particular, $T^1\Omega_{i,N-i}$ is closed for $1\leq i\leq N-1$,
\item $T^1\Omega_{N-1,N-1}$ is open and dense in $T^1\Omega$.
\end{itemize}
When $\K=\R$ we compute $\dim(T^1\Omega_{i,j})=i(N-i)+\frac{i(i+1)}{2}+j(N-j)+\frac{j(j+1)}{2}-1$.

We denote by $\Geod(\Omega)_{i,j}$ the quotient $T^1\Omega_{i,j}/(\phi_t)_{t\in\R}$; observe that $\Geod(\Omega):=T^1\Omega/(\phi_t)_{t\in\R}$ identifies with the set of pairs $(x,y)$ in $\partial\Omega^2$ such that $\Ker(x)\cap\Ker(y)=\emptyset$. We are going to prove that $\NW(\Geod(\Omega),G)$ is the union $\bigcup_{1\leq i\leq N-1}\Geod(\Omega)_{i,N-i}$. This exactly means, according to Section~\ref{The non-wandering set}, that $\NW(T^1\Omega,G\times\R)$ is $\bigcup_{1\leq i\leq N-1}T^1\Omega_{i,N-i}$. We choose a basepoint $v_{i,N-i}\in T^1\Omega_{i,N-i}$, such that $\pi v_{i,N-i}$, $\phi_{-\infty}v_{i,N-i}$ and $\phi_\infty v_{i,N-i}$ are the projectivisations of, respectively, the identity matrix, the orthogonal projection onto $\K^i\times\{0\}$ and the orthogonal projection onto $\{0\}\times\K^{N-i}$. We set 
\[A_{i,N-i}:=\left\{a_t:=\left[\begin{matrix} e^{t/2}I_i & 0 \\ 0 & e^{-t/2}I_j \end{matrix} \right] : t\in\R\right\}\subset A,\]
where $I_k$ is the identity matrix of size $k$, and we observe that for any $t\in\R$, the image $a_t\cdot v_{i,N-i}$ is exactly $\phi_tv_{i,N-i}$. We denote by $G_0$ the identity component of $G$ and by $K_{i,N-i}$ the stabilizer in $G_0$ of $v_{i,N-i}$; they are normalised by $A_{i,N-i}\subset G_0$.

\begin{prop}\label{NW dans le cas symétrique}
Consider $N\geq 3$, the algebra $\K=\R$, $\C$, $\mathbb{H}$, or $\mathbb{O}$ (if $N=3$), the vector space $V=V_{N,\K}$, the properly convex open set $\Omega=\Omega_{N,\K}\subset\PR(V)$, and the group $G=G_{N,\K}$, with identity component~$G_0$.
\begin{enumerate}[label=(\arabic*)]
 \item \label{item:T1Oij vide} For $1\leq i,j\leq N-1$, the set $T^1\Omega_{i,j}$ is non-empty if and only if $i+j\geq N$.
 \item \label{item:alg interpretation of geodflow} For $1\leq i\leq N-1$, the group $G_0$ acts transitively on $T^1\Omega_{i,N-i}$. If we identify $T^1\Omega_{i,N-i}$ with $G_0/K_{i,N-i}$, then the geodesic flow identifies with the action by right multiplication of $A_{i,N-i}$ on $G_0/{K_{i,N-i}}$.
 \item The non-wandering set of $G$ on $\Geod(\Omega)$ is
 \[\NW(\Geod(\Omega),G) = \bigcup_{1\leq i\leq N-1}\Geod(\Omega)_{i,N-i}.\]
\end{enumerate}
\end{prop}

\begin{proof}
\begin{enumerate}[label=(\arabic*)]
 \item Suppose there exists $v\in T^1\Omega_{i,j}$. Because $G_0$ acts transitively on $\Omega$ we can find $g_1\in G_0$ such that $g_1\pi v$ is the projectivisation of the identity matrix. Then by the spectral theorem there exists an automorphism $g_2\in K$ (\ie fixing $g_1\pi v$) such that $\Ker(g_2g_1\phi_{-\infty}v)=\{0\}\times\K^{N-i}$; since the space of $(N-i)$-dimensional right $\K$-sub-modules of $V$ is connected, we can take $g_2$ in $G_0$. We note that the subspaces $\Ker(g_2g_1\phi_{-\infty}v)$ and $\Ker(g_2g_1\phi_\infty v)$ are orthogonal. (Indeed, if $T$ and $T'$ are representatives in $V$ of $g_2g_1\phi_{-\infty}v$ and $g_2g_1\phi_{\infty}v$ such that $T+T'$ is the identity matrix, and if $x\in\Ker(T)$ while $y\in\Ker(T')$, then $\langle x,y\rangle = \langle x,Ty+T'y\rangle=\langle x,Ty\rangle=\langle Tx,y\rangle=0$.) This implies that $i+j\geq N$.
 
 \item Let $1\leq i\leq N-1$. Let us show that there exists $g\in G_0$ such that $g\cdot v$ is the basepoint $v_{i,N-i}$ of $T^1\Omega_{i,N-i}$. We have already seen that there are $g_1,g_2\in G$ such that $\pi g_2g_1v$ is the projectivisation of the identity matrix, and $\Ker(g_2g_1\phi_{-\infty}v)=\{0\}\times\K^{N-i}$. Then $\Ker(g_2g_1\phi_\infty v)=\K^i\times\{0\}$, since $\Ker(g_2g_1\phi_{-\infty}v)$ and $\Ker(g_2g_1\phi_{\infty}v)$ are orthogonal. Moreover $g_2g_1\phi_{-\infty}v$ and $g_2g_1\phi_\infty v$ are the projectivisations of the orthogonal projections onto $\K^i\times\{0\}$ and $\{0\}\times\K^{N-i}$. (Indeed, consider representatives $T$ and $T'$ of $\phi_{-\infty}g_2g_1v$ and $\phi_\infty g_2g_1v$ in $V$ such that $T+T'$ is the identity matrix; then $T'$ and $T$ are the orthogonal projections onto $\K^i\times\{0\}$ and $\{0\}\times\K^{N-i}$.)
 
 \item The stabiliser of $(\phi_{-\infty}v_{i,N-i},\phi_\infty v_{i,N-i})\in\Geod(\Omega)_{i,N-i}$ contains $A_{i,N-i}$, therefore the stabilisers of points in $\Geod(\Omega)_{i,N-i}$ are non-compact, hence $\bigcup_{1\leq i\leq N-1}\Geod(\Omega)_{i,N-i}$ is contained in $\NW(\Geod(\Omega),G)$. Let us prove the converse inclusion.

Suppose by contradiction that $\NW(\Geod(\Omega),G)$ is not contained in $\bigcup_{1\leq i\leq N-1}\Geod(\Omega)_{i,N-i}$. We may assume the existence of sequences of  positive semi-definite Hermitian matrices $(S_n)_{n\in\N}$, $(T_n)_{n\in\N}$ in $V$, of automorphisms $(g_n)_{n\in\N}$ in $\Aut(C)$ and of positive scalars $(\mu_n)_{n\in\N},(\nu_n)_{n\in\N}$, such that
\begin{itemize}
\item $(S_n)_{n}$, $(\mu_n g_nS_n)_{n}$, $(T_n)_{n}$, and $(\nu_n g_nT_n)_{n}$ respectively converge to $S$, $S'$, $T$, and $T'$,
\item the rank of $S$ and $S'$ is $i$, the rank of $T$ and $T'$ is $j$, with $1\leq i,j\leq N-1$ and $i+j>N$,
\item $\Ker(S)\cap\Ker(T)=\Ker(S')\cap\Ker(T')=\{0\}$,
\item $([g_n])_n\in G^\N$ leaves every compact subset of $G$.
\end{itemize}
 Using $\Aut(C)=KAK$ and extracting, we may assume (up to renormalising) that $g_n=a_n\in A$ converge in $\End(V)$ to a non-invertible non-zero  diagonal matrix $a$ with non-negative coefficients. We extend to $a$ the action of $A$ on $V$, with the same notation: $a\cdot X=aXa$ for any $X\in V$.

Since $\Ker(S)\cap\Ker(T)=\{0\}$, up to exchanging $S$ and $T$, we can assume that the image of $a$ is not contained in $\Ker(S)$, and this implies that $a\cdot S\neq 0$. Both $(a_n\cdot S_n)_n$ and $(\mu_na_n\cdot S_n)_n$ converge to a non-zero element of $V$, so $(\mu_n)_n$ must be bounded, and we may assume that it converges to $1$, without loss of generality. Therefore, $a\cdot S=S'$, which means the rank of $a$ is bounded below by $i$. Since $i+j>N$, the kernel of $a$ is not contained in $\Ker(T)$, and $a\cdot T\neq 0$. As before, without loss of generality, we can assume that $a\cdot T=T'$. But now the kernel of $a$ is contained in $\Ker(S')\cap\Ker(T')=\{0\}$, this is a contradiction.\qedhere
\end{enumerate}
\end{proof}

\subsection{The non-wandering set of \texorpdfstring{$(\phi_t)_{t\in\R}$}{phit} on \texorpdfstring{$T^1M$}{T1M}}

Let $\Gamma$ be a lattice of $G$, not necessarily uniform. We set $M=\Omega/\Gamma$. 

\begin{rqq}\label{T1Mbip vide cassym}
 The biproximal unit tangent bundle $T^1M_{bip}$ is empty. To see this, recall that the attracting fixed point of a proximal automorphism of $\Omega$ is always an extremal point of $\Omega$, so by definition the proximal limit set of $\Gamma$ is contained in the closure of the set of extremal points of $\Omega$. Here, since $\Omega$ is symmetric, the set of extremal points is closed and consists of projectivisations of rank-1 positive semi-definite Hermitian matrices, so the set of straight geodesics between two extremal points is $\Geod(\Omega)_{1,1}$, which is empty since $N\geq 3$. 
\end{rqq}

For $1\leq i,j\leq N-1$, we denote by $T^1M_{i,j}$ the quotient $T^1\Omega_{i,j}/\Gamma$. In this section we use the following celebrated theorem of Howe--Moore to study the action of the geodesic flow on each $T^1M_{i,N-i}$, with $1\leq i\leq N$. 

\begin{fait}[\!\cite{howe_moore}, see \eg {\cite[Th.\,2.2.20]{ergo_semi_simple}}]\label{Howe--Moore}
 Let $G$ be a connected non-compact simple Lie group with finite center, let $\pi$ be a unitary representation of $G$ in a separable Hilbert space, without any non-zero $G$-invariant vector. Let $x,y$ be two vectors in the Hilbert space. Then $\langle x,gy \rangle$ converges to zero when $g$ goes to infinity, \ie $g$ leaves every compact subset of $G$.
\end{fait}

Proposition~\ref{NW dans le cas symétrique} and Section~\ref{The non-wandering set} imply that the sets $\NW(\Geod(\Omega),\Gamma)$, $\NW(T^1\Omega,\Gamma\times\R)$, and $\NW(T^1M,(\phi_t)_{t\in\R})$ are respectively contained in $\bigcup_{1\leq i\leq N-1}\Geod(\Omega)_{i,N-i}$, $\bigcup_{1\leq i\leq N-1}T^1\Omega_{i,N-i}$ and $\bigcup_{1\leq i\leq N-1}T^1M_{i,N-i}$. We are now going to see that we actually have equalities. Recall that a finite measure $\mu$ preserved by a measurable flow $(\phi_t)_{t\in\R}$ is called \emph{mixing} if, for any two functions $f,g\in \mathrm{L}^2(\mu)$ with zero integral, we have
\[\int f\cdot (g\circ \phi_t)\, d \mu \underset{t\to\infty}{\longrightarrow} 0.\]
Recall also that a continuous flow is topologically mixing on the support of a mixing invariant measure. Therefore Proposition~\ref{symmetric_top_mixing} is an immediate consequence of the following proposition, and of Zimmer's rigidity theorem \cite[Th.\,1.4]{zimmer_higher_rank}.

\begin{prop}\label{symmetric_mixing}
 Consider $N\geq 3$, the algebra $\K=\R$, $\C$, $\mathbb{H}$, or $\mathbb{O}$ (if $N=3$), the vector space $V=V_{N,\K}$, the properly convex open set $\Omega=\Omega_{N,\K}$, and the group $G=G_{N,\K}$. Take a lattice $\Gamma$ of $G$, not necessarily uniform, and denote by $M$ the quotient $\Omega/\Gamma$. Then for any $1\leq i\leq N-1$, the (finite and fully supported) Haar measure on $T^1M_{i,N-i}$ is mixing under the geodesic flow; as a consequence the geodesic flow is topologically mixing on $T^1M_{i,N-i}$. Furthermore, $\NW(T^1M,(\phi_t)_{t\in\R})$ has exactly $N-1$ connected components, which are $\{T^1M_{i,N-i} : i=1,\dots,N-1\}$.
\end{prop}

\begin{proof}
 Up to replacing $\Gamma$ by a finite-index subgroup, we can assume that $\Gamma$ is contained in $G_0$. Since $\Gamma$ is a lattice, the Haar measure $m$ on $\Gamma\backslash G_0$ is finite. Fix $1\leq i\leq N-1$. By applying the Howe--Moore theorem (Fact~\ref{Howe--Moore}) to the unitary representation of $G_0$ in $\mathrm{L}^2(\Gamma\backslash G_0,m)$, we obtain that $m$ is mixing under the action of the (one-parameter) non-compact subgroup $A_{i,N-i}$ of $G_0$. According to Proposition~\ref{NW dans le cas symétrique}.\ref{item:alg interpretation of geodflow}, it immediately follows that the induced measure on $T^1M_{i,N-i}=\Gamma\backslash G/K_{i,N-i}$ is mixing under the action of the geodesic flow. Since the Haar measure in fully supported, the geodesic flow on $T^1M_{i,N-i}$ is topologically mixing, and its non-wandering set is $T^1M_{i,N-i}$.
\end{proof}

\begin{appendices}

\section{Proof of Crampon's Lemma~\ref{crampon}}\label{appendix}

It is enough to establish Lemma~\ref{crampon} when $c_1(0)=c_2(0)$. Indeed, suppose the lemma true in this case. Consider two straight geodesics $c_1$ and $c_2$, each parametrised with constant speed. Let $c_3$ be the straight geodesic, parametrised with constant speed, such that $c_3(0)=c_1(0)$ and $c_3(T)=c_2(T)$. For $t\leq T$ we have
\begin{align*}
d_\Omega(c_1(t),c_2(t)) &\leq d_\Omega(c_1(t),c_3(t)) + d_\Omega(c_3(t),c_2(t))\\
&\leq d_\Omega(c_1(T),c_3(T))+d_\Omega(c_3(0),c_2(0))\\
&\leq d_\Omega(c_1(T),c_2(T))+d_\Omega(c_1(0),c_2(0)).
\end{align*}

We now assume $c_1(0)=c_2(0)$ (and that $c_1$ and $c_2$ are not constant, otherwise the proof is trivial). We can then assume that $\Omega$ has dimension 2, and we can consider an affine chart in which both projective lines $(c_1(-\infty)\oplus c_2(-\infty))$ and $(c_1(\infty)\oplus c_2(\infty))$ are vertical. Fix $0<t<T$. We draw Figure~\ref{figure_crampon} (left-hand side) which contains the following points:

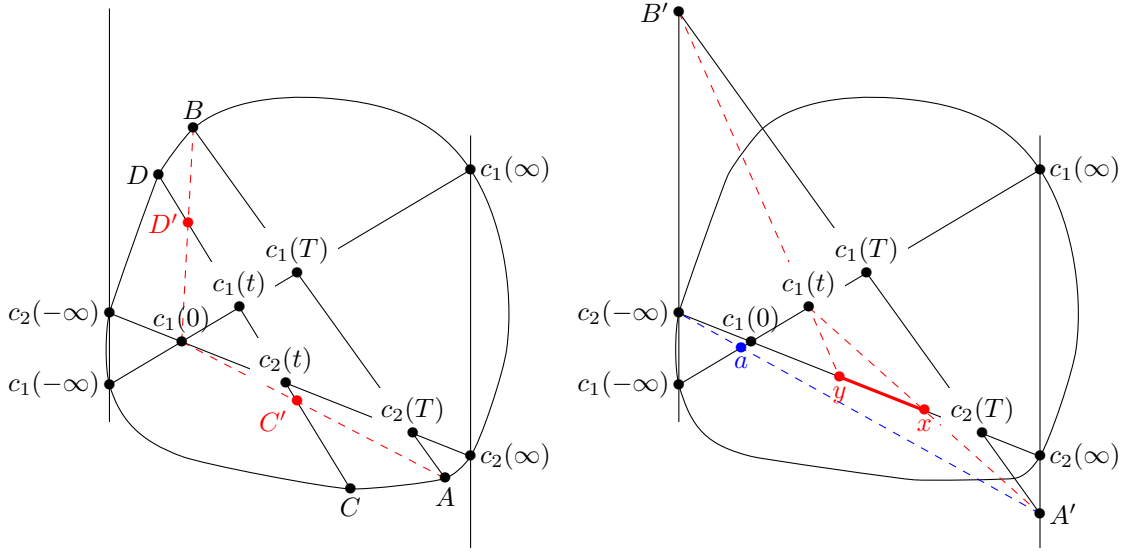
\begin{figure}
\begin{center}
\begin{tikzpicture}[scale=.95]
\coordinate (c1minfty) at (0,2);
\coordinate (c1infty) at (5,5);
\coordinate (c2minfty) at (0,3);
\coordinate (c2infty) at (5,1);
\coordinate (e1) at (1,6);
\coordinate (e2) at (3,6);
\coordinate (e3) at (5.5,2.5);
\coordinate (e4) at (3.5,0);
\coordinate (e5) at (1,1);

\coordinate (c0) at (intersection of c1minfty--c1infty and c2minfty--c2infty);
\coordinate (c1T) at ($ (c0)!0.4!(c1infty) $);
\coordinate (c2T) at ($ (c0)!0.8!(c2infty) $);
\coordinate (c1t) at ($ (c0)!0.5!(c1T) $);
\coordinate (c2t) at ($ (c0)!0.45!(c2T) $);
\coordinate (A') at (intersection of c1T--c2T and c1infty--c2infty);
\coordinate (B') at (intersection of c1T--c2T and c1minfty--c2minfty);
\coordinate (x) at (intersection of A'--c1t and c2minfty--c2infty);
\coordinate (y) at (intersection of B'--c1t and c2minfty--c2infty);
\coordinate (a) at (intersection of A'--c2minfty and c1minfty--c1infty);
\coordinate (b) at (intersection of B'--c2infty and c1minfty--c1infty);
\coordinate (A) at ($ (c2T)!-0.28!(c1T) $);
\coordinate (B) at ($ (c1T)!-0.9!(c2T) $);
\coordinate (C) at ($ (c2t)!-1.4!(c1t) $);
\coordinate (D) at ($ (c1t)!-1.75!(c2t) $);
\coordinate (C') at (intersection of C--D and c0--A);
\coordinate (D') at (intersection of C--D and c0--B);

\coordinate (v1) at ($ (c1infty)!-0.1!(A') $);
\coordinate (v2) at ($ (A')!-0.1!(c1infty) $);
\coordinate (v3) at ($ (B')!-0.01!(c1minfty) $);
\coordinate (v4) at ($ (c1minfty)!-0.1!(B') $);

\length{e5, C, A, c2infty, e3, c1infty, e2, B, D, c2minfty, c1minfty}
\cvxx{e5, C, A, c2infty, e3, c1infty, e2, B, D, c2minfty, c1minfty}{1.8}

\draw (c1minfty)--(c1infty);
\draw (c2minfty)--(c2infty);
\draw (v1)--(v2);
\draw (v3)--(v4);
\draw (C)--(D);
\draw (A)--(B);
\draw [dashed,red] (c0)--(A);
\draw [dashed,red] (c0)--(B);

\draw (c1minfty) node{$\bullet$} node[left]{$c_1(-\infty)$};
\draw (c2minfty) node{$\bullet$} node[left]{$c_2(-\infty)$};
\draw (c1infty) node{$\bullet$} node[right]{$c_1(\infty)$};
\draw (c2infty) node{$\bullet$} node[right]{$c_2(\infty)$};
\draw (c1T) node[above,fill=white]{$c_1(T)$} node{$\bullet$};
\draw (c2T) node[above,fill=white]{$c_2(T)$} node{$\bullet$};
\draw (c0) node{$\bullet$} node[above]{$c_1(0)$};
\draw (c1t) node[above,fill=white]{$c_1(t)$} node{$\bullet$};
\draw (c2t) node[above,fill=white]{$c_2(t)$} node{$\bullet$};
\draw (A) node[below]{$A$} node{$\bullet$};
\draw (B) node[above]{$B$} node{$\bullet$};
\draw (C) node[below]{$C$} node{$\bullet$};
\draw (D) node[left]{$D$} node{$\bullet$};
\draw (C') node[below left,red]{$C'$} node[red]{$\bullet$};
\draw (D') node[left,red]{$D'$} node[red]{$\bullet$};
\end{tikzpicture}
\begin{tikzpicture}[scale=.95]
\coordinate (c1minfty) at (0,2);
\coordinate (c1infty) at (5,5);
\coordinate (c2minfty) at (0,3);
\coordinate (c2infty) at (5,1);
\coordinate (e1) at (1,6);
\coordinate (e2) at (3,6);
\coordinate (e3) at (5.5,2.5);
\coordinate (e4) at (3.5,0);
\coordinate (e5) at (1,1);

\coordinate (c0) at (intersection of c1minfty--c1infty and c2minfty--c2infty);
\coordinate (c1T) at ($ (c0)!0.4!(c1infty) $);
\coordinate (c2T) at ($ (c0)!0.8!(c2infty) $);
\coordinate (c1t) at ($ (c0)!0.5!(c1T) $);
\coordinate (c2t) at ($ (c0)!0.45!(c2T) $);
\coordinate (A') at (intersection of c1T--c2T and c1infty--c2infty);
\coordinate (B') at (intersection of c1T--c2T and c1minfty--c2minfty);
\coordinate (x) at (intersection of A'--c1t and c2minfty--c2infty);
\coordinate (y) at (intersection of B'--c1t and c2minfty--c2infty);
\coordinate (a) at (intersection of A'--c2minfty and c1minfty--c1infty);
\coordinate (b) at (intersection of B'--c2infty and c1minfty--c1infty);
\coordinate (A) at ($ (c2T)!-0.28!(c1T) $);
\coordinate (B) at ($ (c1T)!-0.9!(c2T) $);
\coordinate (C) at ($ (c2t)!-1.28!(c1t) $);
\coordinate (D) at ($ (c1t)!-1.75!(c2t) $);
\coordinate (C') at (intersection of C--D and c0--A);
\coordinate (D') at (intersection of C--D and c0--B);

\coordinate (v1) at ($ (c1infty)!-0.1!(A') $);
\coordinate (v2) at ($ (A')!-0.1!(c1infty) $);
\coordinate (v3) at ($ (B')!-0.01!(c1minfty) $);
\coordinate (v4) at ($ (c1minfty)!-0.1!(B') $);

\length{e5, C, A, c2infty, e3, c1infty, e2, B, D, c2minfty, c1minfty}
\cvxx{e5, C, A, c2infty, e3, c1infty, e2, B, D, c2minfty, c1minfty}{1.8}

\draw (c1minfty)--(c1infty);
\draw (c2minfty)--(c2infty);
\draw (A')--(B');
\draw [red,dashed] (c1t)--(A');
\draw [red,dashed] (B')--(y);
\draw [blue,dashed] (A')--(c2minfty);
\draw (v1)--(v2);
\draw (v3)--(v4);
\draw [red,very thick] (x)--(y);

\draw (c1minfty) node{$\bullet$} node[left]{$c_1(-\infty)$};
\draw (c2minfty) node{$\bullet$} node[left]{$c_2(-\infty)$};
\draw (c1infty) node{$\bullet$} node[right]{$c_1(\infty)$};
\draw (c2infty) node{$\bullet$} node[right]{$c_2(\infty)$};
\draw (c1T) node[above,fill=white]{$c_1(T)$} node{$\bullet$};
\draw (c2T) node[above,fill=white]{$c_2(T)$} node{$\bullet$};
\draw (c0) node{$\bullet$} node[above]{$c_1(0)$};
\draw (c1t) node[above,fill=white]{$c_1(t)$} node{$\bullet$};
\draw (A') node{$\bullet$} node[right]{$A'$};
\draw (B') node{$\bullet$} node[left]{$B'$};
\draw (x) node[red]{$\bullet$} node[below,red]{$x$};
\draw (y) node[red]{$\bullet$} node[below,red]{$y$};
\draw (a) node[blue]{$\bullet$} node[below,blue]{$a$};
\end{tikzpicture}

\end{center}
\caption{Proof of Crampon's Lemma~\ref{crampon}}\label{figure_crampon}
\end{figure}

\begin{itemize}
\item $A$ and $B$ are the intersection points of the line $(c_2(T)\oplus c_1(T))$ with $\partial\Omega$;
\item $C$ and $D$ are the intersection points of the line $(c_2(t)\oplus c_1(t))$ with $\partial\Omega$;
\item $C'$ and $D'$ are the intersection points of the line $(c_2(t)\oplus c_1(t))$ with the lines $(c_1(0)\oplus A)$ and $(c_1(0)\oplus B)$.
\end{itemize}

If we are in the case, as in Figure~\ref{figure_crampon}, where the lines $(c_1(t)\oplus c_2(t))$ and $(c_1(T)\oplus c_2(T))$ do not intersect inside $\Omega$, then by convexity of $\Omega$ the point $C'$ lies between $C$ and $c_2(t)$ and $D'$ lies between $D$ and $c_1(t)$. Therefore by definition of the cross-ratio we deduce that
\begin{align*}
d_\Omega(c_1(T),c_2(T)) &= d_{(C',D')}(c_1(t),c_2(t))\\
&\geq d_{(C,D)}(c_1(t),c_1(t))\\
&\geq d_\Omega(c_1(t),c_2(t)).
\end{align*}

It remains to prove that the lines $(c_1(t)\oplus c_2(t))$ and $(c_1(T)\oplus c_2(T))$ do not cross inside $\Omega$ (this is the missing explanation in Crampon's original proof). We draw Figure~\ref{figure_crampon} (right-hand side) which contains the points:

\begin{itemize}
\item $A'$ and $B'$ are the intersection points of the line $(c_2(T)\oplus c_1(T))$ with the lines $(c_1(\infty)\oplus c_2(\infty))$ and $(c_1(-\infty)\oplus c_2(-\infty))$.
\item $x$ and $y$ are  the intersection points of the line $(c_2(-\infty)\oplus c_2(\infty))$ with the lines $(c_1(t)\oplus A')$ and $(c_1(t)\oplus B')$.
\item $a$ is the intersection point of the line $(c_1(-\infty)\oplus c_1(\infty))$ with the line $(c_2(-\infty)\oplus A')$.
\end{itemize}
And we observe that it is enough to prove that $c_2(t)$ is on the segment $[x,y]$. In other words we want to establish:

\begin{align*}
\frac{d_\Omega(c_2(0),y)}{d_\Omega(c_2(0),c_2(T))}\leq \frac{d_\Omega(c_2(0),c_2(t))}{d_\Omega(c_2(0),c_2(T))} =\frac{t}{T}= \frac{d_\Omega(c_1(0),c_1(t))}{d_\Omega(c_1(0),c_1(T))} \leq \frac{d_\Omega(c_2(0),x)}{d_\Omega(c_2(0),c_2(T))}.
\end{align*}
For example, if we want to establish the inequality on the right, we see by definition of the cross-ratio that it is enough to prove:

\[ \frac{d_\Omega(c_1(0),c_1(t))}{d_\Omega(c_1(0),c_1(T))} \leq \frac{d_{(a,c_1(\infty))}(c_1(0),c_1(t))}{d_{(a,c_1(\infty))}(c_1(0),c_1(T))}.\]
It is a consequence of the following lemma. This, and a similar argument for the inequality on the left, conclude the proof of Lemma~\ref{crampon}.

\begin{lemma} For all $a<a'<x<y<z<b\in\R$,
\[\frac{d_{(a,b)}(x,y)}{d_{(a,b)}(x,z)}\leq \frac{d_{(a',b)}(x,y)}{d_{(a',b)}(x,z)}.\]
\end{lemma}

\begin{proof}
 Up to acting by a projective transformation we can assume that $x=0$, $y=1$ and $b=\infty$. For $z>1$ we consider the function:
 \[a\mapsto f_z(a) = \frac{d_{(a,\infty)}(0,1)}{d_{(a,\infty)}(0,z)}.\]
 on $(-\infty,0)$. We have to check that this function $f_z$ is non-decreasing. This follows immediatly from the fact that, for every $a<0$,
 \[f_z(a)=\frac{\log(1+\frac{-1}{a})}{\log(1+\frac{-z}{a})}.\]
 and from the computation of the derivative.
\end{proof}
\end{appendices}

\bibliographystyle{alpha}
{\small \bibliography{bib}}

\Addresses

\end{document}